\documentclass{article}

\usepackage{amsmath,amsthm} 
\usepackage{amssymb,mathrsfs} 
\usepackage{a4wide} 
\usepackage{graphicx}
\usepackage{physics}
\usepackage{color,subfigure} 
\usepackage{enumerate}
\usepackage{todonotes}
\usepackage[normalem]{ulem}
\usepackage{cancel}

\newtheorem{theorem}{Theorem}
\newtheorem{lemma}{Lemma}
\newtheorem{assumption}{Assumption}
\newtheorem{definition}{Definition}
\newtheorem{prop}{Proposition}
\newtheorem{remark}{Remark}
\newtheorem{corollary}{Corollary}

\DeclareMathAlphabet{\mathpzc}{OT1}{pzc}{m}{it}
\newcommand{\dps}{\displaystyle } 
\newcommand{\rme}{\mathrm{e}}
\newcommand{\ri}{\mathrm{i}} 
\newcommand{\cL}{\mathcal{L}}
\newcommand{\cLs}{\mathcal{S}}
\newcommand{\cLa}{\mathcal{A}}
\newcommand{\Schur}{\mathfrak{S}_0}
\newcommand{\Schurb}{\mathfrak{S}_1}
\newcommand{\cLham}{\mathcal{L}_{\rm ham}}
\newcommand{\cLFD}{\mathcal{L}_{\rm FD}}
\newcommand{\cB}{\mathcal{B}}

\newcommand{\cX}{\mathcal{X}}
\newcommand{\cD}{\mathcal{D}}
\newcommand{\eps}{\varepsilon}
\newcommand{\R}{\mathbb{R}}
\newcommand{\Id}{\mathrm{Id}} 

\newcommand{\cR}{\mathcal{R}}
\newcommand{\cH}{\mathcal{H}}

\newcommand{\subplus}{\textnormal{\texttt{+}}}
\renewcommand{\leq}{\leqslant}
\renewcommand{\geq}{\geqslant}
\renewcommand{\le}{\leqslant}
\renewcommand{\ge}{\geqslant}

\begin{document}

  
\title{Hypocoercivity with Schur complements}
\author{E. Bernard,$^1$ M. Fathi,$^2$ A. Levitt$^{3,1}$ and G. Stoltz$^{1,3}$ \\
  \small 1: CERMICS, Ecole des Ponts, Marne-la-Vallée, France \\
  \small 2: LPSM and LJLL, Universit\'e de Paris, France \\
  \small 3: MATHERIALS project-team, Inria Paris, France
}
 
\maketitle

\begin{abstract}
  We propose an approach to obtaining explicit estimates on the resolvent of
  hypocoercive operators by using Schur complements, rather
  than from an exponential decay of the evolution
  semigroup combined with a time integral. We present applications to Langevin-like dynamics and
  Fokker--Planck equations, as well as the linear Boltzmann equation
  (which is also the generator of randomized Hybrid Monte Carlo in
  molecular dynamics). In particular, we make precise the dependence
  of the resolvent bounds on the parameters of the dynamics and on the
  dimension. We also highlight the relationship of our method with
  other hypocoercive approaches.
\end{abstract}

\section{Introduction}

Degenerate dynamics appear in many contexts. Two prominent examples
are Langevin-type dynamics in molecular dynamics (whose associated generators are Fokker--Planck operators), and Boltzmann-type
equations in the kinetic theory of fluids. From an analytical viewpoint,
these operators have a degenerate dissipative structure. For instance, Fokker--Planck operators
are partial differential operators with
degenerate second order derivatives, in contrast to generators associated with non-degenerate diffusive dynamics which have full second order  derivatives. The key point in proving the longtime convergence of degenerate dynamics is to retrieve some dissipation in all degrees of freedom by a combination of the transport part of the evolution and the degenerate diffusion -- as provided by the hypocoercive techniques reviewed below.

We first briefly describe the two models which motivate our study
(Langevin-type dynamics and linear Boltzmann equation), and then
provide a review of existing hypocoercive approaches. We then turn to
our motivation, namely providing bounds on the resolvent of the
partial differential operator under consideration directly, without
going through kinetic estimates on the evolution semigroups; and
describe our approach based on Schur complements.
  
\paragraph{Paradigmatic hypocoercive dynamics.}
Our first motivation for studying hypocoercive operators stems from
molecular dynamics~\cite{FrenkelSmit,Tuckerman,LM15,LS16}, the
computational armed wing of statistical physics~\cite{Balian}. One
family of dynamics commonly used to compute average properties according to the
Boltzmann--Gibbs measure are Langevin dynamics. These dynamics evolve the positions~$q$ and momenta~$p$ of given particles interacting via a potential energy function~$V(q)$ according to the following stochastic differential equation:
\[
\left\{
\begin{aligned}
  dq_t & = p_t \, dt,\\
  dp_t & = -\nabla V(q_t) \, dt - \gamma p_t \, dt + \sigma \, dW_t, 
\end{aligned}
\right.
\]
where $W_t$ is a standard Brownian motion and $\gamma,\sigma$ are
positive numbers. The first result on the convergence of Langevin-like
dynamics to the stationary measure is~\cite{Tropper77}, which however does not provide explicit convergence rates. The generator associated with Langevin dynamics
\[
\cL_{\rm Lang} = p \cdot \nabla_q - \nabla V(q) \cdot \nabla_p - \gamma p \cdot \nabla_p + \frac{\sigma^2}{2} \Delta_p
\]
is degenerate since second derivatives in~$q$ are missing.

Another classical example of hypocoercive equation is the linear Boltzmann equation modelling the behaviour of material particles interacting with the environment and subjected to a potential~$V$. The behavior of the particles is described in an average way by the density~$f(t,y,v)$ which gives the probability to observe a particle at position~$y$ with velocity~$v$ at time~$t$. This density evolves as
\begin{equation}
  \label{eq:Boltzmann_eq}
  \frac{\partial f}{\partial t}(t,y,v) - \nabla V(y) \cdot \nabla_v f(t,y,v) + v\cdot\nabla_y f(t,y,v) = (M f)(t,y,v).
\end{equation}
Here, $M$ is the collision operator modelling the interactions between the particles and the environment, which leads to some dissipative structure in the velocities. In its simplest form, it is constructed from an integral operator with kernel~$k$:
$$
(M f)(t,y,v) = \int_{\R^d}k(y,v,w)f(t,y,w) \, dw -\sigma (y,v)f(t,y,v),
$$
with the equilibrium condition
\[
\int_{\R^d}k(y,v,w)\, dv = \sigma (y,w),
\]
meaning that there is neither creation nor annihilation of particles.
A special case is to choose for $k(y,v,w)$ a positive multiple of a
Gaussian density with identity covariance in the~$v$ variable
(independent of $y,w$), which corresponds in molecular dynamics to the
generator of the so-called randomized Hybrid Monte Carlo
method~\cite{BRSS17}. The operator
\[
  \cL_{\rm Boltz} = \nabla V\cdot \nabla_v-v\cdot\nabla_y + M
\] 
generates a strongly continuous semigroup in $L^{2}$, and under various conditions on~$k$ and~$\sigma$, it decays exponentially fast towards a global equilibrium in long time. This exponential decay arises from some coupling between the Hamiltonian transport part and the dissipation in the velocities. Besides the hypocoercive techniques reviewed below, let us mention other convergence results specific to this equation, based on spectral methods~\cite{Ukai74, Vidav70}, control theory~\cite{HKL15} or the theory of positive semigroups (see~\cite{BS13,Mokhtar14} and references therein). However, these methods are not constructive and therefore do not provide explicit quantitative estimates of the exponential rate of decay. This is in contrast with hypocoercive approaches, which allow to obtain quantitative bounds by a clever combination of the transport and dissipation part of the dynamics. 

\paragraph{A literature review on convergence results for operators with degenerate dissipation.}
Let us list various approaches to study the exponential convergence of
semigroups associated with hypoelliptic or degenerate generators, in a
somewhat chronological order (see also the recent
review~\cite{Herau16}):
\begin{itemize}
\item A first set of results is based on Lyapunov
  techniques~\cite{Wu01,MSH02,rey-bellet}, the typical Lyapunov
  function being the total energy of the system plus some term
  coupling positions and momenta. This approach was mostly used for Fokker--Planck operators
  such as the generators of Langevin dynamics, but it can also be used for specific Boltzmann-type
  operators~\cite{BRSS17,CELMM, CCEY19}. 
  The corresponding convergence rates
  are however usually not very explicit in term of the parameters of
  the dynamics, because it is difficult to quantify the constants appearing in the so-called minorization condition in terms of the parameters of the dynamics (such as the friction $\gamma$ for Fokker--Planck operators associated with Langevin dynamics). There were however recent attempts at providing more quantitative minorization conditions, using tools from Malliavin calculus~\cite{Ev18}. 
  
\item Subelliptic estimates~\cite{EH03,HN04,HN05} allow to obtain detailed information on the spectrum of Fokker--Planck operators (discrete nature, localization in a cusp region). The spectrum is completely known for systems with zero potential on a torus~\cite{Kozlov89}, or for quadratic potential energy functions (see~\cite{MPP02,PavliotisBook}).

\item $H^1$ hypocoercivity was pioneered in~\cite{Talay02,MN06} and later abstractified in~\cite{Villani09}. The application of this theory to Langevin dynamics allows to quantify the convergence rates in terms of the parameters of the dynamics; see for instance~\cite{HP08} for the Hamiltonian limit and~\cite{LMS16,LS16} for the overdamped limit. The method can be extended to Generalized Langevin dynamics~\cite{OP11} and certain piecewise deterministic Markov processes (whose generators are a specific form of Boltzmann operators)~\cite{DPD18}. It can also be used to study the discretization of Fokker--Planck equations~\cite{DHL19,Georg18}. 
 
 Exponential convergence  in $H^1$ can be established using only $L^2$ bounds on the initial data by hypoelliptic regularization~\cite{Herau07}. The latter technique can also be adapted to discretization schemes~\cite{PZ17}.

\item Entropic estimates for Fokker--Planck operators were initiated in~\cite{DV01} and abstractified in~\cite{Villani09}, though under conditions stronger than those for $H^1$. Recently, it was shown in~\cite{CGMZ19} how to remove the assumption that the Hessian of the potential is bounded. Entropic hypocoercivity for the linear Boltzmann equation was studied in \cite{Ev17, Mon17}, and for linearized BGK models (that generalize the linear Boltzmann equation) in spatial dimension one in~\cite{AAC16}.
  
\item A more direct route to proving the convergence in $L^2$ was first proposed in~\cite{Herau06}, then extended in~\cite{DMS09,DMS15}, and revisited in~\cite{GS16} where domain issues of the operators at play are addressed. This method can be applied to Fokker--Planck operators and Boltzmann-type operators. It is based on a modification of the $L^2$ scalar product with some regularization operator. This more direct approach makes it even easier to quantify convergence rates; see~\cite{DKMS13,GS16} for studies on the dependence of parameters such as friction in Langevin dynamics, as well as~\cite{AAS15,ASS20} for sharp estimates for equilibrium Langevin dynamics and a harmonic potential energy function.

The approach can be perturbatively extended to nonequilibrium situations~\cite{BHM17,SV18,IOS19} and to discretizations of the generator, either via spectral methods~\cite{RS18} or with finite volume schemes~\cite{BCHR19}. It also allows to consider non-quadratic kinetic energies (for which the associated generator may fail to be hypoelliptic)~\cite{ST18}. It can be applied to various dynamics, such as Adaptive Langevin~\cite{LSS19} or certain piecewise deterministic Markov processes~\cite{ADNR18}. Finally, a combination of this approach with a mode-by-mode analysis in Fourier space was introduced in \cite{BDMMS17} to establish polynomial decay estimates for dynamics without confinement, and extended in \cite{BDLS19} to cover fractional diffusion operators. Alternatively, it is also possible to rely on weak Poincar\'e estimates to obtain convergence rates when there is no confinement~\cite{GW19}.

\item It was shown recently how to use techniques from~$\Gamma_2$ calculus for degenerate operators corresponding to Langevin dynamics~\cite{Baudoin17, Mon19a}. This approach allows to study more degenerate dynamics corresponding to the evolution of chains of oscillators~\cite{Menegaki19}.
  
\item Fully probabilistic techniques, based on clever coupling strategies, can also be used to obtain the exponential convergence of the law of Langevin processes to their stationary state~\cite{EGZ19}. One interest of this approach is that the drift needs not be gradient, in contrast to standard analytical approaches for which the analytical expression of the invariant measure should be known in order to separate the symmetric and antisymmetric parts of the generator under consideration. 
This coupling approach can combined with ideas from $\Gamma_2$ calculus \cite{CG14}. 

\item Finally, it was recently shown how to directly obtain $L^2$ estimates without changing the scalar product, relying on a space-time Poincar\'e inequality to conclude to an exponential convergence in time of the evolution semigroup. The method was initially developped for the generators of Langevin dynamics~\cite{AM19,CLW19}, but can be extended to certain Boltzmann-type operators~\cite{LW20} (see Appendix~\ref{sec:gen_Poincare} below for a more detailed account).
\end{itemize}
Note that degenerate norms can be considered for approaches based on $H^1$, $L^2$ or Wasserstein norms, as initially done in~\cite{Talay02}, and recently used in~\cite{Baudoin17,OL15,IOS19}. It is also possible to extend various convergence results to singular interaction potentials~\cite{HM19,Herzog18,BGH19,CHSG21}, and systems with several conservation laws~\cite{CDHMMS21}. Let us finally mention that some approaches are related one to another, such as Lyapunov techniques and estimates based on Poincar\'e inequalities~\cite{BCG08}.

\paragraph{Motivations.}
Our aim in this work is to directly obtain bounds on the resolvent~$\cL^{-1}$, as for subelliptic estimates~\cite{EH03,HN04,HN05}, without going through kinetic estimates on the semigroup $\rme^{t \cL}$ which usually require a modified scalar product (often involving a small parameter which makes the final estimates on the resolvent less quantitative). The motivation for obtaining such resolvent bounds stems from the fact they are a key element for proving that a Central Limit Theorem holds for time averages along realizations of Langevin-like dynamics, and for obtaining estimates on the associated asymptotic variance; see for example~\cite{Bhattacharya} which shows that a sufficient condition for the Central Limit Theorem to hold in this context is that $-\cL^{-1}$ is well defined on a subspace of~$L^2(\mu)$, where~$\mu$ is the invariant measure of the underlying stochastic dynamics. Note that it is also possible to use hypocoercive results to obtain non-asymptotic concentration inequalities for time averages beyond the Central Limit Theorem~\cite{BRB19,Mon19b}.

From a more algebraic perspective, our motivation was to understand the origin of the expression of the regularization operator which appears in the modified $L^2$ scalar product in~\cite{Herau06,DMS09,DMS15}, as well as the algebraic manipulations of~\cite{AM19,CLW19}, which bear some similarity with some computations in~\cite{Herau06,DMS09,DMS15}. Our hope is to extract a structure as general as possible, which encompasses many hypocoercive dynamics such as underdamped Langevin with non-quadratic kinetic energies~\cite{ST18}, the linear Boltzmann equation or randomized Hybrid Monte Carlo with jump processes on the momenta~\cite{BRSS17}, Adaptive Langevin dynamics with a dynamical friction~\cite{Herzog18,LSS19}, Langevin dynamics with extra non-reversible perturbation~\cite{DNS17}, etc.

\paragraph{Saddle-point problems and Schur complements.}
We now motivate the technical approach we use, which is based on Schur complements once the operator under consideration has been written in a form reminiscent of a saddle-point problem. We illustrate the idea on a very simple example of stationary equation displaying some hypocoercive structure, namely the stationary Stokes equation:
\begin{align*}
  \nabla\cdot u &= 0, \\
  -\alpha \Delta u + \nabla p &= f,
\end{align*}
where $\alpha > 0$ is a viscosity constant, $u : \Omega \to \R^{d}$ is
the velocity field and $p : \Omega \to \R$ is the pressure field, with
$\Omega$ a subset of $\R^{d}$ ($d = 2$ or $3$). For simplicitly, we
take for $\Omega$ the torus $(\R/(2\pi\mathbb{Z}))^{d}$ with periodic
boundary conditions, and add the constraints
$\int_{\Omega} p =0, \int_{\Omega} u = 0$. In a vector representation,
the equation can be rewritten as
$\mathcal L_{\rm Stokes} (p,u) = (0,f)$, with
\begin{align*}
  \mathcal{L}_{\rm Stokes} = 
  \begin{pmatrix}
    0 & \nabla^{T}\\
    \nabla & -\alpha \Delta
  \end{pmatrix}.
\end{align*}
Expanding $p,u$ and $f$ into Fourier series, the equation is decoupled
for each mode $n \in \mathbb Z^{d}$, resulting in the
$(1+d)$-dimensional algebraic equation
$\mathcal L_{{\rm Stokes},n} (p_{n},u_{n}) = (0,f_{n})$, where
\begin{align*}
  \mathcal{L}_{{\rm Stokes},n} =
  \begin{pmatrix}
    0 & \ri n^{T}\\
    \ri n & \alpha |n|^{2}\, {\rm Id}_{d}
  \end{pmatrix}.
\end{align*}
This $(1+d)\times (1+d)$ matrix is coercive in the last $d$ variables
(associated with the velocity field~$u$), but not in the first one
(associated with the pressure field $p$). However, we can first use
the second equation to solve~$u_{n}$ as a function of~$p_{n}$:
\begin{align*}
  u_{n} = \alpha^{-1} |n|^{-2}(f_{n} - \ri n  p_{n}).
\end{align*}
Next, replacing in the first equation, we obtain
\begin{align*}
  \mathfrak{S}_n p_{n} = \ri n \cdot (\alpha^{-1} |n|^{-2} f_{n}),
\end{align*}
where
$\mathfrak{S}_n = \ri n^{T} \alpha^{-1} |n|^{-2} \ri n = -\alpha^{-1}$
is the \textit{Schur complement}. This shows that
$\mathcal{L}_{{\rm Stokes},n}$ is invertible for all
$n \in \mathbb{Z} \backslash \{0\}$; together with the zero-mean
constraints on $u$ and $p$, this implies that the original problem has
a unique solution. This scheme of proof generalizes to other types of
boundary conditions or equations (see for instance
\cite[Section 2.4]{ern2013theory}).

The Stokes equation is an example of a more general class of problems
characterized by the block operator structure
\begin{align}
\label{eq:block_intro}
  L = \begin{pmatrix}
    0 & B^{T}\\
    -B & A
  \end{pmatrix}.
\end{align}
This type of problems arises for instance as the Hessian of the
Lagrangian of a minimization problem with linear constraints. The
expression of~$L$ corresponds to a splitting of the variables as~$(u,p)$, where~$u$ are the primal variables and~$p$ the dual variables
(Lagrange multipliers). In this
case, solutions are saddle points of the Lagrangian, hence the general
name of \textit{saddle-point problems} (see~\cite{benzi2005numerical}
for a comprehensive review). As for the Stokes equation above, when
$A$ is invertible, one can solve the equation for the second variable
as a function of the first; in the finite-dimensional case, this shows
that the matrix~$L$ is invertible if and only if the Schur complement
$\mathfrak{S} = B^{T} A^{-1} B$ is, which is equivalent to $B$ having full column
rank.

The main result of this paper is an infinite-dimensional
generalization of this simple approach that covers the case of
Fokker--Planck and Boltzmann-type operators. In these applications,
the block structure corresponds to the separation between the
indirectly dissipated variables (position) and the directly dissipated
ones (momentum). The algebraic structure is the same as the one above
(compare~\eqref{eq:block_intro} and~\eqref{eq:saddle_point_like}), but the analytical structure
(bounds on operators) that we use is adapted to the specific case of
kinetic operators.

\paragraph{Outline of the work.}
From a technical point of view, our approach based on Schur complements allows to obtain a formal expression of the resolvent from a decomposition of the operator under study according to the kernel and image of the degenerate symmetric part of the operator, which is at the origin of the degenerate dissipation. To our knowledge, Schur complements were not used to understand the behavior of hypocoercive operators, although they are well known in the mathematical physics literature where they are used for spectral analysis (see for instance the review~\cite{SZ07}). As already hinted at above, the main benefit of our approach is that it avoids the appearance of somewhat uncontrolled prefactors in the bounds on the resolvent, in contrast to estimates obtained from the exponential decay of the evolution semigroup; see in particular Theorem~\ref{thm:bounds_Linv} for general estimates, and for instance Corollary~\ref{cor:Lang} and Proposition~\ref{prop:RHMC} for illustrative applications of this general result. 

We start by presenting the method in an abstract setting in Section~\ref{sec:abstract}, and then apply it to various dynamics in Section~\ref{sec:applications}. We discuss some relationship with the recent works~\cite{AM19,CLW19} in Appendix~\ref{sec:gen_Poincare}.

\paragraph{Extensions and future work tracks.}
This work calls for various extensions and refinements, on which we are currently working, including in particular: 
\begin{enumerate}[(i)]
    \item obtaining resolvent estimates for $(z -\cL)^{-1}$, for negative real parts of~$z$ sufficiently small compared to the imaginary values of~$z$ (as already obtained with subelliptic estimates~\cite{EH03,HN04,HN05}); 
    \item extending the approach to more degenerate dynamics, starting with the Generalized Langevin dynamics~\cite{OP11}, and if possible addressing oscillator chains.
    \end{enumerate}

\section{Abstract resolvent estimates}
\label{sec:abstract}

\subsection{Schur decomposition of the generator}

We consider the Hilbert space~$L^2(\mu)$ for some Boltzmann--Gibbs probability measure~$\mu(dx)$ on a configuration space~$\cX = \R^D$ (the analysis could however be straightforwardly extended to situations such as $\cX = \mathbb{T}^D$ with $\mathbb{T} = \mathbb{R} / \mathbb{Z}$). The space~$L^2(\mu)$ is equipped with its canonical scalar product, denoted by~$\langle\cdot,\cdot\rangle$. The induced norm is denoted by~$\|\cdot\|$. We consider a stochastic dynamics with generator
\begin{align*}
  \cL = \cLs + \cLa,
\end{align*}
seen as an unbounded operator on~$L^2(\mu)$, where
\[
  \cLs = \frac{\cL+\cL^*}{2}, \qquad \cLa = \frac{\cL-\cL^*}{2}
\]
are respectively the symmetric (dissipation) and antisymmetric (transport) parts of~$\cL$. Adjoint operators are considered with respect to the scalar product~$\langle\cdot,\cdot\rangle$. We assume that the space $\mathscr{C}$ of smooth and uniformly bounded functions with uniformly bounded derivatives is a core for both $\cLa$ and~$\cLs$. In the following, with some abuse of notation, we denote by~$\cLa$ and $\cLs$ the closure of these operators.

We assume that~$\mu$ is invariant by the dynamics, which translates into:
\[
\forall \varphi \in \mathscr{C}, \qquad \int_\cX \cL \varphi \, d\mu = 0.
\]
We also assume that $\cL\mathbf{1} = 0$ (which for instance is obviously the case for generators associated with diffusion processes; note that for hypocoercive dynamics, the assumption of the existence of an invariant probability measure implies that no blowup can occur). We define the subspace of functions of~$L^2(\mu)$ with average~0 with respect to~$\mu$:
\[
\cH = \left\{ \varphi \in L^2(\mu) \, \middle| \, \int_\cX \varphi(x) \, \mu(dx) = \left\langle \varphi,\mathbf{1}\right\rangle = 0 \right\}.
\]
In the orthogonal decomposition $L^2(\mu) = \mathrm{Ran}(\mathbf{1}) \oplus \cH$, the operator $\cL$ can therefore be written as
\[
  \cL = \begin{pmatrix} 0 & 0 \\
  0 & \cL|_\cH \end{pmatrix},
\]
where $\cL|_\cH$ denotes the restriction of~$\cL$ to~$\cH$. In particular, the inverse of~$\cL$ cannot be defined on the whole space~$L^2(\mu)$, but only on~$\cH$. In the sequel, we work on the Hilbert space $\cH$ only. We still denote by~$\cL,\cLs,\cLa$ the restrictions of~$\cL,\cLs,\cLa$ to~$\cH$.

Our first structural assumption is the following.
\begin{assumption}
\label{ass:structural_Pi}
There exists an orthogonal projector~$\Pi_{0}$ such that $\Pi_{0}
\mathscr{C} \subset \mathscr{C}$ and
  \begin{equation}
  \label{eq:conditions_L_Pi}
  \Pi_{0} \cLa \Pi_{0} = 0, \qquad \cLs \Pi_{0} = \Pi_{0} \cLs = 0.
\end{equation}
\end{assumption}
We denote by $\Pi_{\subplus} = 1 - \Pi_{0}$ the orthogonal projector complementary to~$\Pi_{0}$, and
by
\begin{align*}
  \cH = \cH_0 \oplus \cH_\subplus
\end{align*}
the resulting decomposition of the Hilbert space. In the sequel, when
$T$ is an operator on $\cH$ and $\alpha,\beta$ labels, we use the notation
$\cH_\alpha = {\rm Ran}(\Pi_\alpha)$ for the subspaces of $\cH$ and
\begin{align*}
  T_{\alpha\beta} =  \Pi_{\alpha} T \Pi_{\beta}:\, \cH_{\beta} \to \cH_{\alpha}, 
\end{align*}
for the restrictions (blocks) of $T$. When~$T_{\alpha\beta}$ is a bounded operator, we denote its operator norm by
\[
\left\|T_{\alpha\beta}\right\| = \sup_{f \in \cH_\beta \backslash \{0\}} \frac{\|T_{\alpha\beta}f\|}{\|f\|}.
\]

In view of Assumption~\ref{ass:structural_Pi}, the operator~$\cL$ can be written on $\cH = \cH_{0} \oplus \cH_{\subplus}$ as
\begin{equation}
  \label{eq:saddle_point_like}
  \cL = \begin{pmatrix} 0 & \cL_{0\subplus}\\
    \cL_{\subplus0} & \cL_{\subplus\subplus} \end{pmatrix}
 = \begin{pmatrix} 0 & \cLa_{0\subplus}\\
    \cLa_{\subplus0} & \cL_{\subplus\subplus} \end{pmatrix}
\end{equation}
with $\cLa_{0\subplus} = -\cLa_{\subplus0}^{*}$. We next \emph{formally} compute the action of $\cL^{-1}$ by solving the following linear system for a given couple $(\phi_{0},\phi_{\subplus}) \in \cH_0 \times \cH_\subplus$:
\[
\begin{pmatrix} 0 & \cLa_{0\subplus}\\
  \cLa_{\subplus0} & \cL_{\subplus\subplus} \end{pmatrix} \begin{pmatrix} u_{0} \\ u_{\subplus} \end{pmatrix} =  \begin{pmatrix} \phi_{0} \\ \phi_{\subplus} \end{pmatrix}.
\]
The second line leads to $u_{\subplus} = \cL_{\subplus\subplus}^{-1}(\phi_{\subplus} - \cLa_{\subplus0} u_{0})$, and then the first line
to $u_{0} = (\cLa_{0\subplus} \cL_{\subplus\subplus}^{-1} \cLa_{\subplus0})^{-1}(-\phi_{0} + \cLa_{0\subplus} \cL_{\subplus\subplus}^{-1} \phi_{\subplus})$. Therefore,
\begin{equation}
  \label{eq:formal_action_resolvent}
  \cL^{-1} =
  \begin{pmatrix}
    \Schur^{-1} & -\Schur^{-1} \cLa_{0\subplus} \cL_{\subplus\subplus}^{-1}\\
    -\cL_{\subplus\subplus}^{-1} \cLa_{\subplus0} \Schur^{-1}  & \cL_{\subplus\subplus}^{-1} + \cL_{\subplus\subplus}^{-1} \cLa_{\subplus0} \Schur^{-1} \cLa_{0\subplus} \cL_{\subplus\subplus}^{-1}
  \end{pmatrix}
\end{equation}
where 
\begin{align*}
  \Schur = -\cLa_{0\subplus} \cL_{\subplus\subplus}^{-1} \cLa_{\subplus0} = \cLa_{\subplus0}^{*} \cL_{\subplus\subplus}^{-1} \cLa_{\subplus0} 
\end{align*}
is the \textit{Schur complement} associated with the decomposition~\eqref{eq:saddle_point_like}. The subscript~0 in~$\Schur$ emphasizes that this operator acts on~$\cH_0$.

To give a meaning to the computations leading to~\eqref{eq:formal_action_resolvent}, we need to ensure that both $\cL_{\subplus\subplus}$ and~$\Schur$ are invertible, and that the operators appearing in~\eqref{eq:formal_action_resolvent} are bounded. We start with conditions ensuring the invertibility of~$\cL_{\subplus\subplus}$.
\begin{assumption}
  \label{ass:A_inv}
  There exists $s > 0$ such that $-\cLs \ge s \Pi_{\subplus}$ in the sense of symmetric operators.
\end{assumption}
This shows in particular that $-\cL_{\subplus\subplus}$ is coercive on~$\cH_\subplus$, and therefore invertible, with
\begin{align*}
  \left\| \cL_{\subplus\subplus}^{-1} \right\| \leq s^{-1}.
\end{align*}
It is clear from~\eqref{eq:saddle_point_like} that if $\cLa_{\subplus0}$ has a non trivial kernel, then $\cL$ cannot be invertible. We therefore need a quantitative assumption about the injectivity of this operator (a property called ``macroscopic coercivity'' in~\cite{DMS15}).
\begin{assumption}[Macroscopic coercivity]
  \label{ass:macro_coercivity}
  There exists $a > 0$ such that
  \begin{equation}
    \label{eq:macro_coercivity_bound}
    \forall \varphi \in \cH_{0} \cap \mathscr{C}, \qquad \|\cLa_{\subplus0} \varphi\| \geq a \|\varphi\|.
  \end{equation}
\end{assumption}

Let us next introduce the range of~$\cLa_{\subplus 0}$:
\begin{align*}
  \cH_{1} = \mathrm{Ran}\left(\cLa_{\subplus 0}\right) = \overline{\{\cLa_{\subplus0} \varphi, \ \varphi \in \cH_{0} \cap \mathscr{C} \}}^{\cH_{\subplus}}.
\end{align*}
Assumption~\ref{ass:macro_coercivity} ensures that $-\cLa_{0\subplus} \cLa_{\subplus0} = \cLa_{\subplus0}^{*} \cLa_{\subplus0} \geq a^2$, so that $-\cLa_{0\subplus} \cLa_{\subplus0}$ is invertible on~$\cH_{0}$. It is then easily seen that the operator
\begin{equation}
  \label{eq:Pi_1}
  \Pi_{1} = \cLa_{\subplus0} \left(\cLa_{\subplus0}^*\cLa_{\subplus0}\right)^{-1} \cLa_{\subplus0}^*
\end{equation}
is an orthogonal projector, and that its range is $\cH_{1}$, so that $\Pi_{1}$ is the orthogonal projector onto~$\cH_{1}$.

We further decompose~$\cH_\subplus$ orthogonally as $\cH_\subplus = \cH_1 \oplus \cH_2$, which boils down to defining $\cH_2 = (\cH_0 \oplus \cH_1)^\perp$. We denote by~$\Pi_{2}$ the orthogonal projector onto~$\cH_2$. We therefore have the following decomposition of~$\cH$:
\begin{align*}
  \cH = \cH_{0} \oplus \underbrace{\cH_{1} \oplus \cH_{2}}_{=\cH_{\subplus}},
\end{align*}
which induces the following decomposition of the generator:
\[
  \cL = \begin{pmatrix} 0 & \cLa_{01} & 0 \\
    \cLa_{10} & \cL_{11} & \cL_{12} \\
    0 & \cL_{21} & \cL_{22}
  \end{pmatrix},
  \qquad
  \cLa_{01} = -\cLa_{10}^{*}.
\]
The operator $\cLa_{10}:\cH_0\to\cH_1$ is injective by
Assumption~\ref{ass:macro_coercivity}, and surjective by definition
of~$\cH_{1}$. It is therefore invertible, and the norm of its
inverse~$\cLa_{10}^{-1}$ is bounded by~$1/a$ in view
of~\eqref{eq:macro_coercivity_bound}. It can be easily checked that the explicit expression of the inverse is
\begin{align*}
  \cLa_{10}^{-1} = (\cLa_{10}^{*}\cLa_{10})^{-1} \cLa_{10}^{*} = (\cLa_{\subplus0}^{*}\cLa_{\subplus0})^{-1} \cLa_{10}^{*}.
\end{align*}

\begin{remark}
  Note that~$\cLa_{10}^{-1}$ is quite similar to the regularization operator $R = (1+\cLa_{\subplus0}^*\cLa_{\subplus0})^{-1}\cLa_{\subplus 0}^*$ introduced in~\cite{Herau06,DMS09,DMS15} to modify the scalar product on~$L^2(\mu)$ for proving hypocoercive properties. More precisely, the main steps of the proof in~\cite{DMS09,DMS15} are to introduce a modified~$L^2(\mu)$ scalar product, induced by the modified norm~$\langle M_\varepsilon f, f \rangle$ with
  \[
  M_\varepsilon = \begin{pmatrix}
    1 & \varepsilon R\\
    \varepsilon R^* & 1
  \end{pmatrix}
  \]
  and equivalent to the canonical one; and to prove that~$-\cL$ is coercive with respect to this modified scalar product, \emph{i.e.} $\langle -M_\varepsilon\cL f, f\rangle \geq \kappa \langle M_\varepsilon f,f\rangle$ for some~$\kappa>0$. The proof can be extended to any regularization operator of the form $R = (\eta+\cLa_{\subplus0}^*\cLa_{\subplus0})^{-1}\cLa_{\subplus 0}^*$ with~$\eta>0$. The inverse~$\cLa_{10}^{-1}$ is recovered in the limit~$\eta \to 0$. Our algebraic derivation therefore provides a complementary viewpoint on the origin of the regularization operator in~\cite{Herau06,DMS09,DMS15}.
\end{remark}

We next remark that
\begin{equation}
  \label{eq:Schur_L++11}
\Schur = \cLa_{\subplus0}^{*} \cL_{\subplus\subplus}^{-1} \cLa_{\subplus0} = \cLa_{10}^{*} \left[\cL_{\subplus\subplus}^{-1}\right]_{11} \cLa_{10}.
\end{equation}
Both operators $-\cL_{\subplus\subplus}$ and $-\cL_{22}$ are coercive and hence invertible. We can then consider a second Schur complement to compute the action of~$\left[\cL_{\subplus\subplus}^{-1}\right]_{11}$:
\[
\left[\cL_{\subplus\subplus}^{-1}\right]_{11} = \Schurb^{-1}, \qquad \Schurb = \cL_{11} - \cL_{12} \cL_{22}^{-1} \cL_{21}.
\]
Therefore, $\Schur = \cLa_{10}^{*} \Schurb^{-1}\cLa_{10}$, and so
\begin{equation}
  \begin{aligned}
  \label{eq:Schur-1}
  \Schur^{-1}
  & = \cLa_{10}^{-1}\left( \cL_{11} - \cL_{12} \cL_{22}^{-1} \cL_{21} \right)\cLa_{10}^{-*} \\
  & = \left(\cLa_{\subplus0}^*\cLa_{\subplus0}\right)^{-1} \cLa_{10}^* \left( \cL_{11} - \cL_{12} \cL_{22}^{-1} \cL_{21} \right) \cLa_{10} \left(\cLa_{\subplus0}^*\cLa_{\subplus0}\right)^{-1}.
  \end{aligned}
\end{equation}
We introduce a final structural assumption, which simplifies the action of some of the operators.

\begin{assumption}
  \label{ass:reversibility}
  There exists an involution~$\cR$ on~$\cH$ (i.e. a bounded operator for which $\cR^2 = 1$) such that
  \[
    \cR \Pi_{0} = \Pi_{0} \cR = \Pi_{0}, \qquad \cR \cLs\cR  = \cLs, \qquad \cR \cLa \cR = - \cLa. 
  \]
\end{assumption}
This assumption is motivated by the physical notion of microscopic reversibility for kinetic dynamics, where~$\cR$ can be chosen as the momentum-reversal operator (see Section~\ref{sec:applications}). Assumption~\ref{ass:reversibility} implies in particular that
\begin{equation}
  \label{eq:RLR=R*}
  \cR \cL \cR = \cL^*.
\end{equation}
Since $\cR\Pi_\subplus = \cR(1-\Pi_0) = \Pi_\subplus \cR$, it holds
\begin{equation}
  \label{eq:R_A+0}
  \cR \cLa_{\subplus0} = -\cLa_{\subplus0},
\end{equation}
and therefore $\cR \Pi_{1} = -\Pi_{1}$. Similarly, $\Pi_{1} \cR = -\Pi_{1}$. The operator~$\cR$ can therefore be decomposed as
\begin{align*}
  \cR =
  \begin{pmatrix}
    1&0&0\\
    0&-1&0\\
    0&0&\cR_{22}
  \end{pmatrix}.
\end{align*}
It then follows from~\eqref{eq:RLR=R*} that $\cL_{11}$ is symmetric,
\emph{i.e.} $\cL_{11} = \cLs_{11}$. In view of the latter equality and
of \eqref{eq:Schur-1}, we finally assume the following to ensure the invertibility of the Schur complement~$\Schur$.
\begin{assumption}
  \label{ass:technical_ass}
  The operators $\cLs_{11}$ and $\cL_{21}\cLa_{10}\left(\cLa_{\subplus0}^* \cLa_{\subplus0}\right)^{-1}$ are bounded.
\end{assumption}

As we shall see in Section~\ref{sec:applications}, Assumptions~\ref{ass:structural_Pi} to~\ref{ass:technical_ass} are satisfied for generators of Langevin-like dynamics, under suitable conditions on the potential and kinetic energies. The final result can then be summarized as follows.

\begin{theorem}
  \label{thm:bounds_Linv}
  Suppose that Assumptions~\ref{ass:structural_Pi} to~\ref{ass:technical_ass} hold true. Then, $\cL$ is invertible on~$\cH$ and
  \[
  \begin{aligned}
    & \|\cL^{-1}\| \leq 2\left( \frac {\|\cLs_{11}\|} {a^{2}} + \frac {\|\cR_{22}\|\|\cL_{21}\cLa_{10}(\cLa_{\subplus0}^{*}\cLa_{\subplus0})^{-1}\|^{2}}{s} \right) + \frac3s.
    \end{aligned}
  \]
\end{theorem}

\begin{proof}
  From~\eqref{eq:Schur-1} and since $\cL_{11} = \cLs_{11}$, we obtain the following bound:
  \[
    \|\Schur^{-1}\| \le \|\cLa_{10}^{-1} \cLs_{11} \cLa_{10}^{-*}\| + \left\|\left(\cLa_{\subplus0}^*\cLa_{\subplus0}\right)^{-1} \cLa_{10}^{*}\cL_{12} \right\|\left\|\cL_{22}^{-1} \right\|\left\|\cL_{21} \cLa_{10} \left(\cLa_{\subplus0}^*\cLa_{\subplus0}\right)^{-1}\right\|.
  \]
  Since $(\cL_{12})^{*} = (\cL^{*})_{21} = (\cR \cL \cR)_{21} = -\cR_{22} \cL_{21}$,
  \[
    \left\|\left(\cLa_{\subplus0}^*\cLa_{\subplus0}\right)^{-1} \cLa_{10}^{*}\cL_{12} \right\|
    = \left\|\left(\cL_{12}\right)^* \cLa_{10} \left(\cLa_{\subplus0}^*\cLa_{\subplus0}\right)^{-1} \right\|
    \leq \|\cR_{22}\|\left\|\cL_{21} \cLa_{10} \left(\cLa_{\subplus0}^*\cLa_{\subplus0}\right)^{-1}\right\|.
  \]
From Assumption \ref{ass:A_inv} we have $\left\|\cL_{22}^{-1} \right\| \leq s^{-1}$. Therefore,
\begin{equation}
  \label{eq:Schur_inv_bound}
  \|\Schur^{-1}\| \le \frac {\|\cLs_{11}\|} {a^{2}} + \frac {\|\cR_{22}\|\left\|\cL_{21} \cLa_{10}(\cLa_{\subplus0}^{*}\cLa_{\subplus0})^{-1}\right\|^{2}}{s}.
\end{equation}

We now turn to the other terms in \eqref{eq:formal_action_resolvent}.
Note first that $\cR_{\subplus\subplus} \cL_{\subplus\subplus}
\cR_{\subplus\subplus} = \cL_{\subplus\subplus}^*$, so that
$\cR_{\subplus\subplus} \cL_{\subplus\subplus}^{-1}
\cR_{\subplus\subplus} = \cL_{\subplus\subplus}^{-*}$. In view
of~\eqref{eq:Schur_L++11} and~\eqref{eq:R_A+0}, the Schur
complement~$\Schur$ is in fact symmetric and can be written as $\Schur
= \cLa_{\subplus0}^{*} \left[\cL_{\subplus\subplus}^{-1}\right]_{\rm
  s} \cLa_{\subplus0}$, where $T_{{\rm s}} = (T + T^{*})/2$ denotes
the symmetric part of an operator~$T$. Since $[T^{-1}]_{{\rm s}} = T^{-*} T_{{\rm s}} T^{-1}$,
the operator $[-\cL_{\subplus\subplus}^{-1}]_{\rm s} =
-\cL_{\subplus\subplus}^{-*}
\cLs_{\subplus\subplus}\cL_{\subplus\subplus}^{-1}$ is a positive bounded
self-adjoint operator, with a well-defined square root. The Schur
complement is therefore symmetric negative and can be factored as
\begin{align*}
  \Schur = -Q^{*} Q, \qquad Q = [-\cL_{\subplus\subplus}^{-1}]_{{\rm s}}^{1/2} \cLa_{\subplus0}.
\end{align*}
Consider now the second term in the lower right corner of~\eqref{eq:formal_action_resolvent}:
\begin{align*}
  \cL_{\subplus\subplus}^{-1} \cLa_{\subplus0} \Schur^{-1} \cLa_{\subplus0}^{*} \cL_{\subplus\subplus}^{-1} = -\underbrace{\cL_{\subplus\subplus}^{-1} [-\cL_{\subplus\subplus}^{-1}]_{\rm s}^{-1/2}}_{T_{1}} \; \underbrace{Q (Q^{*}Q)^{-1} Q^{*}}_{T_{2}}\; \underbrace{[-\cL_{\subplus\subplus}^{-1}]_{\rm s}^{-1/2}\cL_{\subplus\subplus}^{-1}}_{T_{3}}.
\end{align*}
The operator $T_{2}$ is an orthogonal projector, hence has a norm bounded by~1. We bound~$T_{3}$ by noting that
\begin{align*}
  T_{3}^{*} T_{3} = \cL_{\subplus\subplus}^{-*} [-\cL_{\subplus\subplus}^{-1}]_{\rm s}^{-1} \cL_{\subplus\subplus}^{-1} = -2\left(\cL_{\subplus\subplus}\left[\cL_{\subplus\subplus}^{-1}+\cL_{\subplus\subplus}^{-*}\right]\cL_{\subplus\subplus}^*\right)^{-1} = -2(\cL_{\subplus\subplus}+\cL_{\subplus\subplus}^*)^{-1} = -\cLs_{\subplus\subplus}^{-1},
\end{align*}
and therefore $\|T_{3}\| \le 1/{\sqrt s}$ by Assumption~\ref{ass:A_inv}. Similarly, $\|T_{1}\| \le 1/{\sqrt s}$, so that $\norm{\cL_{\subplus\subplus}^{-1} \cLa_{\subplus0} \Schur^{-1} \cLa_{\subplus0}^{*} \cL_{\subplus\subplus}^{-1}} \le 1/s$. To bound the operator in the upper right corner of~\eqref{eq:formal_action_resolvent}, we remark that $\Schur^{-1} \cLa_{\subplus0}^{*} \cL_{\subplus\subplus}^{-1} = -(Q^{*}Q)^{-1} Q^* T_3$, so that (using $\|A\|^2 = \|A A^*\|$ with $A=(Q^{*}Q)^{-1} Q^{*}$)
\begin{align*}
  \|\Schur^{-1} \cLa_{\subplus0}^{*} \cL_{\subplus\subplus}^{-1}\| \le \frac 1 {\sqrt s} \|(Q^{*}Q)^{-1} Q^*\| = \frac 1 {\sqrt s} \sqrt{\|(Q^{*}Q)^{-1}\|} = \sqrt{\frac { \norm{\Schur^{-1}}} { s}}.
\end{align*}
The same bound holds for the operator in the lower left corner of~\eqref{eq:formal_action_resolvent}.

By gathering all the estimates, we finally obtain
\begin{align*}
  \|\cL^{-1}\| &\le \left\|(\cL^{-1})_{00}\right\| + \left\|(\cL^{-1})_{0\subplus}\right\| + \left\|(\cL^{-1})_{\subplus0}\right\| + \left\|(\cL^{-1})_{\subplus\subplus}\right\|\\
  &\le  \norm{\Schur^{-1}} + 2 \sqrt{\frac { \norm{\Schur^{-1}}} { s}} + \frac 2 s \le 2\norm{\Schur^{-1}} + \frac 3 s,
\end{align*}
which, together with~\eqref{eq:Schur_inv_bound}, gives the desired bound.
\end{proof}

\section{Applications and extensions}
\label{sec:applications}

\subsection{Langevin dynamics}
\label{sec:application_Langevin}

For Langevin dynamics (also known as underdamped Langevin dynamics, or kinetic Langevin dynamics in some communities), the reference measure is the phase-space Boltzmann--Gibbs measure defined on~$\cX = \cD \times \R^d$ with $\cD = \mathbb{T}^d$ or~$\mathbb{R}^d$, which reads
\begin{equation}
    \label{eq:canonical_U_V}
\mu(dq\,dp) = \nu(dq)\,\kappa(dp), \qquad \nu(dq) = Z_\nu^{-1} \rme^{-\beta V(q)} \,dq, \qquad \kappa(dp) = Z_\kappa^{-1} \rme^{-\beta U(p)} \,dp,
\end{equation}
where~$Z_\nu,Z_\kappa \in (0,+\infty)$ are normalization constants ensuring that~$\nu,\kappa$ are probability measures. The dynamics reads, for some positive friction $\gamma > 0$,
\[
\left\{
\begin{aligned}
dq_t & = \nabla U(p_t) \, dt, \\
dp_t & = -\nabla V(q_t) \, dt - \gamma \nabla U(p_t) \, dt + \sqrt{\frac{2\gamma}{\beta}} \, dW_t.
\end{aligned}
\right.
\]
Dynamics in an extended space can also be considered, as we shall see in Section~\ref{sec:AdL}. Note also that while for the classical Langevin dynamics the kinetic energy is simply $U(p) = |p|^2/2$, we consider a general kinetic energy~$U$ in order to emphasize the structure of the dynamics. In particular, $\nabla U$ can vanish on an open set, in which case the generator associated with the corresponding Langevin dynamics is not hypoelliptic~\cite{ST18}. The antisymmetric part of the generator is the generator of the Hamiltonian dynamics 
\begin{equation}
  \label{eq:Lham}
  \cLa = \cLham = \nabla U(p)^T \nabla_q - \nabla V(q)^T \nabla_p,
\end{equation}
while the symmetric part is (the subscript 'FD' stands for 'fluctuation/dissipation') 
\begin{equation}
  \label{eq:LFD_Langevin}
  \cLs = \gamma \cLFD,
  \qquad 
  \cLFD = -\nabla U(p)^T \nabla_p + \frac1\beta \Delta_p = -\frac1\beta \nabla_p^* \nabla_p.
\end{equation}
When the kinetic energy is quadratic, the choice~\eqref{eq:LFD_Langevin} corresponds to considering an Ornstein--Uhlenbeck process on the momenta. The projector$~\Pi_0$ is
\begin{equation}
  \label{eq:def_Pi}
  (\Pi_0 \varphi)(q) = \int_{\R^d} \varphi(q,p)\, \kappa(dp).
\end{equation}
Simple computations show that Assumptions~\ref{ass:structural_Pi} and~\ref{ass:reversibility} hold true when~$U$ is even, upon choosing the momentum reversal operator $\cR f(q, p) = f(q, -p)$ in Assumption~\ref{ass:reversibility}. Note that~$\cR$ is unitary when~$U$ is even, in which case $\|\cR_{22}\| = 1$.

\subsubsection{General estimates}
\label{eq:general_estimates_Langevin}

We assume that $\nu$ and~$\kappa$ satisfy Poincar\'e inequalities. As we will see below, we also need growth conditions on the potential in order to apply some results from~\cite{DMS15}; as well as moment estimates for derivatives of~$U$. To state them, we denote by $\partial_p^\alpha = \partial_{p_1}^{\alpha_1} \dots \partial_{p_d}^{\alpha_d}$ and $|\alpha| = \alpha_1 + \dots + \alpha_d$ for any $\alpha = (\alpha_1,\dots,\alpha_d) \in \mathbb{N}^d$.

\begin{assumption}
  \label{ass:Poincare_nu}
  The function $V$ is smooth and such that $\rme^{-\beta V} \in L^1(\cD)$. There exists $K_\nu > 0$ such that, for any $\phi \in H^1(\nu)$,  
  \[
    \left\| \phi - \int_{\mathcal{D}} \phi \, d\nu \right\|_{L^2(\nu)} \leq \frac{1}{K_\nu} \| \nabla_q \phi  \|_{L^2(\nu)}.
  \]
  Moreover, there exist $c_1 > 0$, $c_2 \in [0,1]$ and $c_3 > 0$ such that
  \begin{equation}
    \label{eq:regularization condition}
    \Delta V \leq c_1 d + \frac {c_2 \beta} 2 | \nabla V |^2, \qquad \left|\nabla^2 V \right|^2 := \sum_{i,j=1}^d \left|\partial_{x_i}\partial_{x_j}V\right|^2 \leq c_3^2 \left( d + | \nabla V |^2 \right),
  \end{equation}
  where~$|\cdot|$ denotes the standard Euclidean norm on~$\R^d$.
\end{assumption}

\begin{assumption}
  \label{ass:Poincare_kappa}
  The function $U$ is smooth and even, and such that $\rme^{-\beta U} \in L^1(\R^d)$. There exists $K_\kappa > 0$ such that, for any $\varphi \in H^1(\kappa)$,  
  \[
    \left\| \varphi - \int_{\R^d} \varphi \, d\kappa \right\|_{L^2(\kappa)} \leq \frac{1}{K_\kappa} \| \nabla_p \varphi  \|_{L^2(\kappa)}.
  \]
  Moreover, the kinetic energy $U$ is such that $\partial_p^\alpha U$ belongs to $L^2(\kappa)$ for any $|\alpha| \leq 3$, and $(\partial_p^{\alpha} U) (\partial_p^{\alpha'} U)$ is in $L^2(\kappa)$ for any~$|\alpha| \leq 2$ and $|\alpha'| = 1$.
\end{assumption}

We refer to~\cite{BBCG} for simple conditions on $U$ and $V$ that ensure that Poincar\'e inequalities hold. The scaling with respect to the dimension of the constants in the bounds~\eqref{eq:regularization condition} is motivated by the case of separable potentials for which $V(q) = v(q_1) + \dots + v(q_d)$ for some smooth one dimensional function~$v$, which corresponds to tensorized probability measures. The bounds~\eqref{eq:regularization condition} then follow from the inequalities
\[
  v'' \leq c_1 + \frac{c_2 \beta}{2} (v')^2, \qquad \left|v''\right|^2 \leq c_3^2 \left( 1+\left|v'\right|^2 \right).
\]
These bounds generally hold if $v$ has polynomial growth for example. The scaling of the constants should be similar for particles on a lattice (such as one dimensional atom chains) with finite interaction ranges, or systems for which correlations between degrees of freedom are bounded with respect to the dimension, in the sense that each column/line of the matrix $\nabla^2 V$ has a finite number of nonzero entries. Note finally that a careful inspection of the proof of Lemma~\ref{lem:VillaniA24} below shows that it would be possible to choose $c_2 \in [0,2)$, but the final estimates would be more cumbersome to write, which is why we stick to the condition $c_2 \in [0,1]$.

\begin{theorem}
  Suppose that Assumptions~\ref{ass:Poincare_nu} and~\ref{ass:Poincare_kappa} hold true. Then, the resolvent of the generator of the Langevin dynamics satisfies the following bound:
\begin{equation}
  \label{eq:bounds_invL_Langevin}
  \begin{aligned}
    & \|\cL^{-1}\| \leq \frac{2\beta\gamma}{\lambda_{\rm min}(\mathcal{M}) K_\nu^2} \|\Pi_1 \cLFD \Pi_1\| \\ 
    & \qquad + \frac{4\beta}{\gamma K_\kappa^2} \left( \frac34 + \left\| \Pi_\subplus \cLham^2 \Pi_0 \left(\cLa_{\subplus 0}^*\cLa_{\subplus 0}\right)^{-1} \right\|^2 + \gamma^2 \left\| \Pi_2 \cLFD \Pi_1 \cLham \Pi_0 \left(\cLa_{\subplus 0}^*\cLa_{\subplus 0}\right)^{-1} \right\|^2 \right),
  \end{aligned}
\end{equation}
where $\lambda_{\rm min}(\mathcal{M}) > 0$ denotes the smallest eigenvalue of the symmetric positive definite matrix
\[
\mathcal{M} = \int_{\R^d} \nabla^2 U \, d\kappa = \beta \int_{\R^d} \nabla U \otimes \nabla U \, d\kappa.
\]
\end{theorem}

Note that $\|\cL^{-1}\|$ indeed scales as $\dps \min(\gamma,\gamma^{-1})$, as noted in various previous works~\cite{DKMS13,GS16,IOS19} (upon possibly integrating in time the bounds on the evolution semigroup).

\begin{proof}
  We apply Theorem~\ref{thm:bounds_Linv}, and check to this end that Assumptions~\ref{ass:A_inv}, \ref{ass:macro_coercivity} and~\ref{ass:technical_ass} hold true. The Poincar\'e inequality in the momentum variable implies that Assumption~\ref{ass:A_inv} holds with
\[
  s = \gamma \frac{K_\kappa^2}{\beta}.
\]
To check Assumptions~\ref{ass:macro_coercivity} and~\ref{ass:technical_ass}, we use $\partial_{p_i}^* = -\partial_{p_i} + \beta \partial_{p_i} U$ and $\partial_{q_i}^* = -\partial_{q_i} + \beta \partial_{q_i} V$ to rewrite the generators as
\[
\cLs = \gamma \cLFD = -\frac \gamma \beta \sum_{i=1}^d \partial_{p_i}^* \partial_{p_i},
\qquad \cLa = \cLham = \frac{1}{\beta} \sum_{i=1}^d \partial_{p_i}^* \partial_{q_i} - \partial_{q_i}^* \partial_{p_i}.
\]
We also use the following rules
\[
[\partial_{p_{i}}, \partial_{q_{j}}] = 0, \qquad \partial_{p_i} \Pi_0 = 0, \qquad \Pi_0 \partial_{p_i}^* = 0, \qquad [\partial_{p_i},\partial_{p_j}^*] = \beta \partial_{p_i,p_j}^2 U,
\]
to obtain 
\begin{equation}
  \label{eq:A^*A_Langevin}
\cLa_{\subplus 0}^*\cLa_{\subplus 0} = \frac 1 {\beta^{2}} \sum_{i=1}^{d} \sum_{j=1}^{d} \Pi_{0} \partial_{q_{i}}^{*} \partial_{p_{i}} \partial_{p_{j}}^{*} \partial_{q_{j}} \Pi_{0}  = \frac1\beta \Pi_{0} \nabla_q^* \mathcal{M} \nabla_q \Pi_0 = \frac1\beta \sum_{i,j=1}^d \mathcal{M}_{ij}  \Pi_{0} \partial_{q_i}^* \partial_{q_j} \Pi_0.
\end{equation}
To prove that $\mathcal{M}$ is a symmetric positive definite matrix, we write, for~$\xi \in \R^d$,
\[
\xi^T \mathcal{M} \xi = \beta \int_{\R^d} \left|\xi^T \nabla U\right|^2 d\kappa. 
\]
If the latter quantity was~0 for some element $\xi \in \R^d \backslash \{0\}$, then $\nabla U$ would be~0 in the direction of~$\xi$, and hence~$U$ would be constant in this direction, which would contradict the assumption~$\rme^{-\beta U} \in L^1(\R^d)$. We then obtain, for a smooth function $\varphi$ with compact support and average~0 with respect to~$\mu$,
\[
  \left\|\cLa_{\subplus 0}\varphi\right\|_{L^2(\mu)}^2 = \frac{1}{\beta}\int_\cD \left(\nabla_q \Pi_0 \varphi\right)^T \mathcal{M} \nabla_q \Pi_0 \varphi \,  d\nu \geq \frac{\lambda_{\rm min}(\mathcal{M})}{\beta} \left\| \nabla_q \Pi_0 \varphi \right\|^2_{L^2(\nu)} \geq \frac{\lambda_{\rm min}(\mathcal{M}) K_\nu^2}{\beta} \|\Pi_0 \varphi\|_{L^2(\mu)}^2,
\]
where we used the Poincar\'e inequality for~$\nu$ in the last step. This shows that Assumption~\ref{ass:macro_coercivity} holds true with
\[
a = K_\nu \sqrt{\frac{\lambda_{\rm min}(\mathcal{M})}{\beta}}.
\]

We finally turn to Assumption~\ref{ass:technical_ass}. Note first that 
\[
-\cLa_{\subplus 0}^* \cLFD \cLa_{\subplus 0} = \frac{1}{\beta} \nabla_q^* \mathcal{N} \nabla_q \Pi_0, \qquad \mathcal{N}_{ij} = \int_{\R^d} \left[\left(\nabla^2 U\right)^2\right]_{ij} d\kappa = \sum_{k=1}^d \int_{\R^d} (\partial^2_{p_i,p_k} U)(\partial^2_{p_k,p_j} U) \,  d\kappa, 
\]
so that, in view of~\eqref{eq:Pi_1} and~\eqref{eq:A^*A_Langevin},
\[
  -\Pi_1 \cLFD \Pi_1 = \frac{1}{\beta}\sum_{i,j=1}^d T_{ij}\widetilde{T}_{ij},
\]
with
\[
  T_{ij} = \partial_{p_i}^* \Pi_0 \partial_{p_j},
  \qquad
  \widetilde{T}_{ij} = \partial_{q_i} \left(\nabla_q^* \mathcal{M} \nabla_q\right)^{-1}\nabla_q^* \mathcal{N} \nabla_q \left(\nabla_q^* \mathcal{M} \nabla_q\right)^{-1} \partial_{q_j}^*,
  \]
  where~$\widetilde{T}_{ij}$ is seen as an operator on~$L^2(\nu)$. Note that $T_{ij}$ is bounded since it is the composition of operators of the form $\Pi_0 \partial_{p_k}$ and their adjoints, such operators being bounded because, by an integration by parts,  
\[
\Pi_0 \partial_{p_k} \varphi = \beta \int_{\mathbb{R}^d} \varphi \, \partial_{p_k}U \, d\kappa,
\]
so that
\[
\left\| \Pi_0 \partial_{p_k} \varphi \right\|_{L^2(\mu)} \leq \beta\|\partial_{p_k} U\|_{L^2(\mu)}\|\varphi\|_{L^2(\mu)}. 
\]
The operators $\widetilde{T}_{ij}$ are also easily seen to be bounded, as a linear combination of composition of operators of the form $\left(\nabla_q^* \mathcal{M} \nabla_q\right)^{-1/2} \partial_{q_j}^*$ (which are bounded from~$L^2(\nu)$ to~$L^2_0(\nu)$, the subspace of functions of functions of~$L^2(\nu)$ with average~0 with respect to~$\nu$) and their adjoints. This allows to conclude that $\cLs_{11} = \gamma \Pi_1\cLFD \Pi_1$ is bounded. 

Let us next consider
\begin{equation}
  \label{eq:op_for_Ass5}
  \cL_{21}\cLa_{10} \left(\cLa_{\subplus 0}^*\cLa_{\subplus 0}\right)^{-1} = \Pi_2 \cLs \cLa \Pi_0 \left(\cLa_{\subplus 0}^*\cLa_{\subplus 0}\right)^{-1} + \Pi_2 \cLa^2 \Pi_0 \left(\cLa_{\subplus 0}^*\cLa_{\subplus 0}\right)^{-1}.
\end{equation}
We start by the second operator on the right hand side of the latter equality. Since $\Pi_2 = \Pi_2 \Pi_\subplus$, it suffices in fact to prove that $\Pi_\subplus\cLa^2 \Pi_0 \left(\cLa_{\subplus 0}^*\cLa_{\subplus 0}\right)^{-1}$ is bounded. First, for a smooth function~$\varphi$ with compact support, we compute 
$\cLa^2\Pi_0\varphi = \nabla U^T(\nabla^2_q \Pi_0 \varphi) \nabla U - \nabla V^T (\nabla^2 U) \nabla_q \Pi_0\varphi$, 
and therefore
\begin{equation}
  \label{eq:action_cLa2}
\Pi_\subplus\cLa^2 \Pi_0 \varphi = \sum_{i,j=1}^d \mathcal{U}_{ij} \partial^2_{q_i,q_j} \Pi_0 \varphi - \mathcal{V}_{ij} \left(\partial_{q_i} V\right) \partial_{q_j} \Pi_0,
\end{equation}
with 
\begin{equation}
  \label{eq:Sij_Tij}
  \mathcal{U}_{ij} = \left[\left(\partial_{p_i} U\right)\left(\partial_{p_j} U\right)-\frac1\beta \mathcal{M}_{ij}\right], 
  \qquad
  \mathcal{V}_{ij} =  \partial^2_{p_i, p_j}U - \mathcal{M}_{ij}.
\end{equation}
The boundedness of the operator $\Pi_\subplus\cLa^2 \Pi_0 \left(\cLa_{\subplus 0}^*\cLa_{\subplus 0}\right)^{-1}$ then follows from the fact that the operators $\Pi_\subplus\partial^2_{q_i,q_j} \Pi_0 \left(\cLa_{\subplus 0}^*\cLa_{\subplus 0}\right)^{-1}$ and $\Pi_\subplus\left(\partial_{q_i} V\right) \partial_{q_j} \Pi_0 \left(\cLa_{\subplus 0}^*\cLa_{\subplus 0}\right)^{-1}$ are bounded (in view of~\eqref{eq:regularization condition}, see respectively~\cite[Proposition~5]{DMS15} and~\cite[Lemma~A.4]{ST18}; see also the proof of Proposition~\ref{prop:estimate_cLa^2_invBB*} below).

Let us finally bound the first operator on the right hand side of~\eqref{eq:op_for_Ass5}, using the following equality as operators from~$\cH_0$ to~$\cH_2$: 
\[
  \begin{aligned}
    \Pi_2 \cLs \cLa \Pi_0 \left(\cLa_{\subplus 0}^*\cLa_{\subplus 0}\right)^{-1} & = -\frac{\gamma}{\beta^2} \sum_{i=1}^d \Pi_2 \mathscr{T}_{i} \partial_{q_i} \left(\cLa_{\subplus 0}^*\cLa_{\subplus 0}\right)^{-1},
\qquad 
\mathscr{T}_i = \sum_{j=1}^d \partial_{p_j}^* \partial_{p_j} \partial_{p_i}^* \Pi_0.
\end{aligned}
\]
The result then follows from the fact that $\partial_{q_i} \left(\cLa_{\subplus 0}^*\cLa_{\subplus 0}\right)^{-1}$ is bounded from~$\cH_0$ to~$\cH$, and
\[
  \mathscr{T}_i = \beta \sum_{j=1}^d \left[\beta \left(\partial_{p_j}U\right)\left(\partial^2_{p_i, p_j}U\right) -\partial^3_{p_i, p_j, p_j} U\right] \Pi_0
\]
is a bounded operator on~$L^2(\mu)$, because of the $L^2(\kappa)$-bounds on the derivatives of $U$ (see Assumption~\ref{ass:Poincare_kappa}). The estimate~\eqref{eq:bounds_invL_Langevin} finally follows by Theorem~\ref{thm:bounds_Linv}.
\end{proof}

\subsubsection{Scaling with the dimension}
\label{sec:scaling_dimension_Langevin}

The aim of this section is to make precise the dependence on the dimension of the various terms in~\eqref{eq:bounds_invL_Langevin}. We consider to this end the simple situation where the kinetic energy is quadratic and where all degrees of freedom are associated with the same mass~$m>0$, \emph{i.e.} 
\begin{equation}
  \label{eq:quadratic_kinetic_energy}
  U(p) = \sum_{i=1}^d u(p_i), \qquad u(p) = \frac{p^2}{2m}.
\end{equation}
In this case, $\kappa(dp)$ is a tensor product of Gaussian measures with variances~$m/\beta$, so that the Poincar\'e constant is~$K_\kappa^2 = \beta/m$. 

\begin{corollary}
  \label{cor:Lang}
  Suppose that Assumption~\ref{ass:Poincare_nu} holds true. Then, the resolvent of the generator of the Langevin dynamics with quadratic kinetic energy~\eqref{eq:quadratic_kinetic_energy} satisfies the following bound:
  \begin{equation}
    \label{eq:bounds_invL_Langevin_quadratic}
    \|\cL^{-1}\| \leq \frac{2\beta\gamma}{K_\nu^2} + \frac{4 m}{\gamma} \left( \frac34 + \left\| \Pi_\subplus \cLham^2 \Pi_0 \left(\cLa_{\subplus 0}^*\cLa_{\subplus 0}\right)^{-1} \right\|^2 \right).
  \end{equation}
\end{corollary}

The operator norm on the right hand side can be bounded using Proposition~\ref{prop:estimate_cLa^2_invBB*} below, or Proposition~\ref{prop:scaling_dimension_Langevin_improved} in Section~\ref{sec:scaling_dimension_Langevin_improved} under some extra assumptions on the potential~$V$. Let us emphasize that a particularly nice feature of the estimate~\eqref{eq:bounds_invL_Langevin_quadratic} is that there is no uncontrolled prefactor in the estimate, as usually obtained by writing the resolvent as the time integral of the semigroup.

\begin{proof}
  First, $\mathcal{M} = m^{-1} \Id$ so that $\lambda_{\rm min}(\mathcal{M}) = m^{-1}$. Consider next $\Pi_1 \cLFD \Pi_1$. Since $\cLFD \cLham \Pi_0 = -\cLham \Pi_0/m$ and therefore $\cLa_{\subplus 0}^{*}\cLFD \cLa_{\subplus 0} = -\cLa_{\subplus 0}^{*}\cLa_{\subplus 0}/m$, it holds $\cLFD\Pi_1 = \Pi_1 \cLFD \Pi_1= -\Pi_1/m$. Since~$\Pi_1$ is an orthogonal projector, we finally obtain $\left\|\Pi_1 \cLFD \Pi_1\right\| \leq m^{-1}$. Moreover, $\Pi_2 \cLFD \Pi_1 = 0$ since $\cLFD \Pi_1 = -\Pi_1/m$ by the above computations. The estimates finally follow from~\eqref{eq:bounds_invL_Langevin}.
\end{proof}

In addition to Assumption~\ref{ass:Poincare_nu}, and similarly to what is done in~\cite{CLW19}, we consider two possible conditions on the potential energy: (i) $V$ is convex; (ii) the Hessian of~$V$ is bounded below as $\nabla^{2}_{q}V \geq -K \Id$ for some $K \in \mathbb{R}_+$. Notice that the second condition implies the first one when $K = 0$. The following result, whose proof is postponed to Section~\ref{sec:proof_prop:estimate_cLa^2_invBB*}, is key in understanding the dependence on the dimension of the bound provided by~\eqref{eq:bounds_invL_Langevin_quadratic}. 

\begin{prop}
  \label{prop:estimate_cLa^2_invBB*}
  Suppose that Assumption~\ref{ass:Poincare_nu} holds. Then, 
\[
\left\| \Pi_\subplus \cLham^2 \Pi_0 \qty(\cLa_{\subplus 0}^*\cLa_{\subplus 0})^{-1} \right\|^2 \leq 2 \qty(C+\frac{C^{'}}{K_\nu^2}),
\]
where $C$ and $C^{'}$ can be chosen as:
\begin{enumerate}[(i)]
\item If $V$ is convex, then $C=1$ and $C^{'}=0$;
\item If $\nabla_q^2 V \geq -K \Id$ for some $K \geq 0$, then $C=1$ and $C^{'}=K$;
\item In the general case, $C=2$ and $\dps C' = 2c_3\left[\sqrt{d}+2\max\left(\frac{8c_3}{\beta^2},\sqrt{\frac{c_1 d}{\beta}}\right) \right]$.
\end{enumerate}
\end{prop}

The scaling with respect to the dimension of the bound for item~(iii) above is better than the one in~\cite{CLW19}, because we directly integrate into our assumption~\eqref{eq:regularization condition} how the prefactors in the estimates should depend on the dimension, and then carefully track this dependency in the proofs of Lemmas~\ref{controlH2} and~\ref{lem:VillaniA24} below (the latter lemma being directly obtained from~\cite[Lemma~A.24]{Villani09}).

\medskip

To finally discuss more precisely the dependence on the dimension of the upper bound~\eqref{eq:bounds_invL_Langevin_quadratic}, we would need to understand how the Poincar\'e constant~$K_\nu$ and constants~$c_1,c_2,c_3$ in Assumption~\ref{ass:Poincare_nu} depend on the dimension:
\begin{enumerate}[(i)]
\item As discussed after Assumption~\ref{ass:Poincare_nu}, $c_1,c_2,c_3$ can be chosen independently of the dimension for simple systems (in particular for separable potentials, in which case the Langevin dynamics can be seen as the juxtaposition of~$d$ one dimensional Langevin dynamics).
\item When $\nabla^2 V \geq R \, \mathrm{Id}$ with $R > 0$, the Poincar\'e constant~$K_\nu^2$ is bounded from below by the positive constant~$\beta R$ by the celebrated Bakry--Emery criterion~\cite{BE85}, independently of the dimension. For separable potentials, the Poincar\'e constant for the full measure is dimension-free as long as the marginal distribution of any single position itself satisfies a Poincar\'e inequality (with some uniform lower bound on the Poincar\'e constant), by the tensorization property. Sufficient conditions for dimension-free constants for functional inequalities when the potentials have weak enough interactions have been well-studied, see for example \cite{yos2001, otto2007}. 
\end{enumerate}
These considerations allow to further make precise the estimates on the resolvent.

\subsubsection{Proof of Proposition~\ref{prop:estimate_cLa^2_invBB*}}
\label{sec:proof_prop:estimate_cLa^2_invBB*}

We start by noting that, for the quadratic kinetic energy~\eqref{eq:quadratic_kinetic_energy}, the expressions~\eqref{eq:action_cLa2}-\eqref{eq:Sij_Tij} simplify as
 \[
\forall 1 \leq i,j \leq d, \qquad \mathcal{U}_{ij}= \frac{p_{i}p_{j}}{m^2} - \frac{1}{m\beta}\delta_{i,j}, \qquad 
\mathcal{V}_{ij}= 0,
\]
and
\[
\Pi_\subplus \cLham^2 \Pi_0 \qty(\cLa_{\subplus 0}^*\cLa_{\subplus 0})^{-1} \varphi = \sum_{i,j=1}^d \mathcal{U}_{ij} \partial^2_{q_i,q_j} \Pi_0 \qty(\cLa_{\subplus 0}^*\cLa_{\subplus 0})^{-1}\varphi.
\]
Besides, a simple computation gives $\|\mathcal{U}_{ij}\|_{L^2(\kappa)}^2 = 1/(m^2\beta^2)$ for $i\neq j$, and $\|\mathcal{U}_{ii}\|_{L^2(\kappa)}^2 = 2/(m^2\beta^2)$, as well as $\left \langle \mathcal{U}_{ij},\mathcal{U}_{kl}\right\rangle_{L^2(\kappa)} = 0$ when $\{i,j\} \neq \{k,l\}$. Therefore, for $\varphi \in L^2(\mu)$, 
\[
  \begin{aligned}
  & \left\|\Pi_\subplus\cLham^2 \Pi_0 \qty(\cLa_{\subplus 0}^*\cLa_{\subplus 0})^{-1} \varphi \right\|^2 = \sum_{i,j=1}^d  \left\|\mathcal{U}_{ij}\right\|_{L^2(\kappa)}^2 \left\| \partial_{q_i,q_j}^2 \Pi_0 \qty(\cLa_{\subplus 0}^*\cLa_{\subplus 0})^{-1}\varphi \right\|^2_{L^2(\nu)} \\
  & \quad\leq \frac{2}{m^2 \beta^2} \sum_{i,j=1}^d \left\| \partial_{q_i,q_j}^2 \Pi_0 \qty(\cLa_{\subplus 0}^*\cLa_{\subplus 0})^{-1}\varphi \right\|^2_{L^2(\nu)} = \frac{2}{m^2 \beta^2} \left\| \nabla_q^2\Pi_0 \qty(\cLa_{\subplus 0}^*\cLa_{\subplus 0})^{-1}\varphi \right\|_{L^2(\nu)}^2, 
\end{aligned}
\]
so that, using $\cLa_{\subplus 0}^*\cLa_{\subplus 0} = (m\beta)^{-1} \nabla_q^*\nabla_q \Pi_0$, 
\[
  \left\|\Pi_\subplus\cLham^2 \Pi_0 \qty(\cLa_{\subplus 0}^*\cLa_{\subplus 0})^{-1} \varphi \right\|^2 \leq 2 \left\| \nabla_q^2 \qty( \nabla_q^*\nabla_q)^{-1}\Pi_0 \varphi \right\|_{L^2(\nu)}^2.
\]
The proof is then directly obtained from the following technical result (where we denote by $C_\mathrm{c}^\infty(\cD)$ the set of $C^\infty$ functions with compact support). 

\begin{lemma}
\label{controlH2}
It holds, with the same constants $C,C^{'} \in \mathbb{R}_+$ as in Proposition~\ref{prop:estimate_cLa^2_invBB*}, 
\[
\forall u\in C_\mathrm{c}^\infty(\cD), \qquad \left\| \nabla^2 u \right\|^2_{L^2(\nu)} = \sum_{i,j=1}^d \left\|\partial_{q_i,q_j}^2 u\right\|^2_{L^2(\nu)} \leq C\|\nabla_q^{*}\nabla_q u\|^2_{L^2(\nu)} + C^{'}\norm{\nabla_{q}u}^2_{L^2(\nu)}.
\]
\end{lemma}


To prove this result, we need the following lemma.

\begin{lemma}
  \label{lem:VillaniA24}
Suppose that Assumption~\ref{ass:Poincare_nu} holds. Then,
\begin{equation}
  \label{Villanilemma}
  \forall \phi \in H^{1}\qty(\nu),
  \qquad
  \left\| \phi \nabla V\right\|^2_{L^2\qty(\nu)} \leq \frac{16}{\beta^2} \|\nabla \phi\|^2_{L^2\qty(\nu)} + \frac{4d c_1}{\beta} \|\phi\|^2_{L^2\qty(\nu)}.
\end{equation}
\end{lemma}

\begin{proof}[Proof of Lemma~\ref{lem:VillaniA24}]
We closely follow the proof of~\cite[Lemma~A.24]{Villani09}, precisely keeping track of all constants (which leads to estimates with a better scaling than in~\cite[Lemma~2.2]{CLW19}). First, for a given function $\varphi \in C_\mathrm{c}^\infty(\cD)$, 
\[
\left\|\phi \nabla V \right\|^2_{L^2\qty(\nu)} =\int_\cD \phi^2 \abs{\nabla V}^2 \, d\nu = \frac1\beta \int_\cD \mathrm{div}\left(\phi^2 \nabla V\right) d\nu = \frac{2}{\beta} \int_\cD \phi \nabla \phi \cdot \nabla V \, d\nu + \frac{1}{\beta}\int_\cD \phi^2 \Delta V \, d\nu,
\]
so that, in view of Assumption~\ref{ass:Poincare_nu}, 
\[
  \norm{ \phi \nabla V}^2_{L^2\qty(\nu)} \leq  \frac{2}{\beta} \norm{\nabla \phi}_{L^2\qty(\nu)} \, \norm{ \phi \nabla V}_{L^2\qty(\nu)} + \frac{c_1 d}{\beta} \norm{ \phi }_{L^2\qty(\nu)}^2 + \frac{c_2}{2} \norm{ \phi \nabla V}_{L^2\qty(\nu)}^2.
\]
By a Young inequality with parameter~$\eta>0$, we then obtain 
\[
\begin{aligned}
\left(1-\frac{c_2}{2}-\frac{\eta}{\beta}\right)\left\| \phi \nabla V\right\|^2_{L^2\qty(\nu)} & \leq \frac{1}{\beta \eta} \norm{\nabla \phi}^2_{L^2\qty(\nu)} + \frac{c_1 d}{\beta} \norm{ \phi }_{L^2\qty(\nu)}^2.
\end{aligned}
\]
Choosing $\eta = \beta/4$ and recalling that $c_2 \in [0,1]$  
provides the claimed estimate by the density of $C_\mathrm{c}^\infty(\cD)$ in~$H^1(\nu)$. 
\end{proof}

We can finally turn to the proof of Lemma~\ref{controlH2}. 

\begin{proof}[Proof of Lemma~\ref{controlH2}]
We fix $u \in C_\mathrm{c}^\infty(\cD)$ and follow the proof of~\cite[Lemma~2.3]{CLW19}, which relies on Bochner's formula 
\[
\sum_{i,j=1}^{d} \abs{\partial_{q_i,q_j}^2 u}^2 = \nabla_q u\cdot \nabla_q \left( \nabla^{*}_q\nabla_q u\right) - \qty(\nabla_q u)^T\nabla_q^2 V \nabla_q u-\nabla_q^{*}\nabla_q\left(\frac{\abs{\nabla_q u}^2}{2}\right),
\]
to write 
\begin{equation}
  \label{Bochnereq}
  \sum_{i,j=1}^{d} \left\|\partial_{q_i,q_j}^2 u\right\|^2_{L^2(\nu)} = \left\|\nabla_q^{*}\nabla_q u\right\|^2_{L^2(\nu)} - \int_\cD \qty(\nabla_q u)^T \nabla_q^2 V \nabla_q u \,  d\nu.
\end{equation}
This already gives the desired result when either~$V$ is convex (\emph{i.e.} $\nabla^2 V \geq 0$) or $\nabla^2 V \geq -K \Id$.

Let us next consider the general case given by Assumption~\ref{ass:Poincare_nu}. Note that $|\nabla^2 V|\leq c_3(\sqrt{d}+|\nabla V|)$. We first obtain from~\eqref{Bochnereq} that
\begin{equation}
  \label{eq:Bochnereq_CS}
  \left\|\nabla_q^2 u\right\|_{L^2(\nu)}^2 \leq \left\|\nabla_q^{*}\nabla_q u \right\|_{L^2(\nu)}^2+ c_{3}\sqrt{d}\|\nabla_q u\|_{L^2(\nu)}^2 + c_{3}\|\nabla_q u\|_{L^2(\nu)}\left\| \left|\nabla_{q}u\right| \, \left| \nabla_{q}V\right| \right\|_{L^2(\nu)}.
\end{equation}
Note that, from~\eqref{Villanilemma},
\[
  \begin{aligned}
    \left\| \left|\nabla_{q}u\right| \, \left| \nabla_{q}V\right| \right\|_{L^2(\nu)}^2 & = \sum_{i=1}^d \left\| \partial_{q_i}u \nabla_{q}V \right\|_{L^2(\nu)}^2 \leq \sum_{i=1}^d \frac{16}{\beta^2} \left\|\nabla \qty(\partial_{q_i}u) \right\|^2_{L^2\qty(\nu)} + \frac{4d c_1}{\beta} \left\|\partial_{q_i}u \right\|^2_{L^2\qty(\nu)} \\
    & \leq \frac{16}{\beta^2} \left\|\nabla_q^2 u \right\|^2_{L^2\qty(\nu)} + \frac{4d c_1}{\beta} \left\|\nabla_q u \right\|^2_{L^2\qty(\nu)}.
  \end{aligned}
\]
Using a Young inequality to bound the last term on the right hand side of~\eqref{eq:Bochnereq_CS} with the inequality above then leads to
\[
  \begin{aligned}
    \left\|\nabla_q^2 u\right\|_{L^2(\nu)}^2 & \leq \left\|\nabla_q^{*}\nabla_q u \right\|_{L^2(\nu)}^2+ c_{3} \left( \sqrt{d}+ \frac{1}{2\eta} \right) \|\nabla_q u\|_{L^2(\nu)}^2 + \frac{c_3 \eta}{2} \left(\frac{16}{\beta^2} \left\|\nabla_q^2u\right\|_{L^2(\nu)}^2 + \frac{4dc_1}{\beta} \|\nabla_q u\|_{L^2(\nu)}^2\right),
  \end{aligned}
\]
and finally
\[
\left(1 - \frac{8c_3 \eta}{\beta^2}\right)\left\|\nabla_q^2 u\right\|_{L^2(\nu)}^2 \leq \left\|\nabla_q^{*}\nabla_q u \right\|_{L^2(\nu)}^2 + c_3\left(\sqrt{d}+\frac{1}{2\eta} + \frac{2\eta dc_1}{\beta} \right) \left\|\nabla_q u\right\|_{L^2(\nu)}^2.
\]
We next choose $\eta = \beta^2\varepsilon/(8c_3)$ for some $\varepsilon \in (0,1)$ to be determined, so that 
\[
  \left\|\nabla_q^2 u\right\|_{L^2(\nu)}^2 \leq \frac{1}{1-\varepsilon} \left[\left\|\nabla_q^{*}\nabla_q u \right\|_{L^2(\nu)}^2 + L_{d,\varepsilon} \left\|\nabla_q u\right\|_{L^2(\nu)}^2\right],
\]
with 
\[
  L_{d,\varepsilon} = c_3\left(\sqrt{d}+\frac{\mathfrak{a}}{\varepsilon} +  \mathfrak{b} \varepsilon \right),
  \qquad
  \mathfrak{a} = \frac{4c_3}{\beta^2}, \qquad \mathfrak{b} = \frac{\beta c_1 d}{4 c_3}.
\]
When~$d$ is large, one should choose $\varepsilon$ small enough in order for $\mathfrak{b}\varepsilon$ to have a limited growth. Such a choice implies that the denominator~$1-\varepsilon$ in the expression of~$L_{d,\varepsilon}$ will be close to~1. In order to balance the diverging term~$\mathfrak{a}/\varepsilon$ as $\varepsilon \to 0$, one should set~$\varepsilon = \sqrt{\mathfrak{a}/\mathfrak{b}}$. We however want to keep $\varepsilon \leq 1/2$ in all cases, which suggests choosing
\[
  \varepsilon = \min\left(\frac12, \sqrt{\frac{\mathfrak{a}}{\mathfrak{b}}}\right),
\]
in which case (using that $\mathfrak{b} \leq 4\mathfrak{a}$ when $\varepsilon = 1/2$)
\[
L_{d,\varepsilon} = c_3\left[\sqrt{d}+2 \max\left(2\mathfrak{a}, \sqrt{\mathfrak{a}\mathfrak{b}} \right)\right] = c_3\left[\sqrt{d}+2\max\left(\frac{8c_3}{\beta^2},\sqrt{\frac{c_1 d}{\beta}}\right) \right].
\]
The conclusion then follows from the fact that $1-\varepsilon \geq 1/2$.
\end{proof}

\subsubsection{Improving the dependence on the dimension under a log-Sobolev inequality}
\label{sec:scaling_dimension_Langevin_improved}

It is possible to improve the explicit dependence on the dimension from~$\sqrt{d}$ to~$\log d$ in item~(iii) of Proposition \ref{prop:estimate_cLa^2_invBB*}, but under some extra assumptions on the potential, and if we strengthen the Poincar\'e inequality assumption into a logarithmic Sobolev inequality.

\begin{definition}
A probability measure~$\nu$ on $\cD$ is said to satisfy a logarithmic Sobolev inequality with constant $C_{\rm LSI}$ if 
\[
\forall f \in C_\mathrm{c}^\infty(\cD), \qquad \operatorname{Ent}_{\nu}(f^2) := \int_{\cD} f^2\log f^2 \, d\nu - \|f\|_{L^2(\nu)}^2\log \|f\|_{L^2(\nu)}^2 \leq C_{\rm LSI}\|\nabla f\|_{L^2(\nu)}^2.
\]
\end{definition}

A Poincar\'e inequality can be derived by linearizing this inequality, so it is strictly stronger than the Poincar\'e inequality we used above. However, it is still valid for uniformly log-concave distributions, and tensorizes with dimension-free constants. We refer to \cite{Led99} for background information, and more general sufficient conditions for LSI to hold. To state the result giving bounds on the resolvent with respect to the dimension, we define, for a matrix $M \in \mathbb{R}^{d \times d}$, the operator norm
\[
\|M\|_{\cB(\ell^2)} = \sup_{|\xi| \leq 1} |M\xi|.
\]
We also denote by $|\xi|_\infty = \max(|\xi_1|,\dots,|\xi_d|)$ the $\ell^\infty$ norm on~$\R^d$.

\begin{prop}
  \label{prop:scaling_dimension_Langevin_improved}
  Consider the quadratic kinetic energy~\eqref{eq:quadratic_kinetic_energy}.
  Assume that $\rme^{-\beta V} \in L^1(\cD)$, that $\nu$ satisfies a logarithmic Sobolev inequality, and that 
  \begin{equation}
    \label{eq:condition_Hessian_LSI}
    \forall x \in \R^d, \qquad \left\|\nabla^2V(x)\right\|_{\cB(\ell^2)} \leq c_3\left(1 + |\nabla V(x)|_{\infty}\right).
  \end{equation}
  Then the conclusion of Proposition \ref{prop:estimate_cLa^2_invBB*} holds true with
  \[
    C = 2, \qquad C' = 2\left(c_3 + \frac{1}{2c_3 C_{\rm LSI}}\left[\log d + \log\left( \max_{i=1,\dots,d} \int{\rme^{2c_3 C_{\rm LSI}|\partial_i V|} \, d\nu}\right)\right]\right).
  \]
\end{prop}

Note that for this proposition to be useful, we implicitly assume the exponential integrability of the partial derivatives of $V$ with respect to $\nu$. If we look at nice enough separable potentials, the dependence of $C'$ on the dimension is of order $\log d$, since the other quantities involved in the bound are dimension-free. Note that a logarithmic scaling in $d$ is typically correct for maximas of~$d$ independent quantities \cite[Lemma 2.2]{dev2001} (which here would correspond to the case of a separable potential). Note also that the assumption~\eqref{eq:condition_Hessian_LSI}, while natural for separable potentials, is not invariant under a change of orthonormal basis.

\begin{proof}
The proof follows the same general approach as that of Proposition~\ref{prop:estimate_cLa^2_invBB*}, but with the extra assumption on $V$, \eqref{eq:Bochnereq_CS} can be replaced by 
\begin{equation}
    \left\|\nabla_q^2 u\right\|_{L^2(\nu)}^2 \leq \left\|\nabla_q^{*}\nabla_q u \right\|_{L^2(\nu)}^2+ c_{3}\|\nabla_q u\|_{L^2(\nu)}^2 + c_{3}\int{|\nabla_q u|^2|\nabla V|_{\infty}d\nu}.
\end{equation}
We cannot use Lemma~\ref{lem:VillaniA24} to control the last term in the above equation, but use instead the following classical entropy inequality: for $g$ bounded measurable, 
\[
\int_\cD f^2g \, d\nu \leq \operatorname{Ent}_{\nu}(f^2) + \|f\|_{L^2(\nu)}^2\log \left( \int_\cD \rme^{g} \, d\nu \right).
\]
This inequality is an immediate consequence of the fact that the entropy is the Legendre transform of the log-Laplace functional, with respect to the duality between functions and measures, \emph{i.e.}
\[
\operatorname{Ent}_{\nu}(f^2) = \sup \left\{ \int_\cD f^2 g \, d\nu - \|f\|_{L^2(\nu)}^2\log \left(\int_\cD \rme^{g} \, d\nu\right), \ \ g \textrm{ bounded measurable} \right\}.
\]
As a consequence, we obtain a bound of the form 
\[
\int_\cD |\nabla_q u|^2|\nabla V|_{\infty} \, d\nu \leq \frac1c \operatorname{Ent}_{\nu}(|\nabla_q u|^2) + \frac1c \|\nabla_q u\|_{L^2(\nu)}^2 \log \left(\int_\cD \rme^{c |\nabla V|_{\infty}} \, d\nu\right),
\]
with $c>0$ a constant to be adjusted later. When $\nu$ satisfies a logarithmic Sobolev inequality, the first factor on the right-hand side can be controlled by $C_{\rm LSI}\|\nabla^2_q u\|^2_{L^2(\nu)}$, using the inequality $\left|\nabla |\nabla f|\right|^2 \leq |\nabla^2 f|^2$. The second factor can be bounded as
\[
\log \left( \int_\cD \rme^{c|\nabla V|_{\infty}} \, d\nu\right) \leq \log \left( \sum_{i=1}^d \int_\cD \rme^{c|\partial_{q_i} V|_{\infty}} \, d\nu\right) \leq \log d + \log\left( \max_{i=1,\dots,d} \int_\cD \rme^{c|\partial_{q_i} V|} \, d\nu \right).
\]
We therefore obtain the inequality
\begin{equation}
\begin{aligned}
    \left\|\nabla^2 u\right\|_{L^2(\nu)}^2 & \leq \left\|\nabla_q^{*}\nabla_q u \right\|_{L^2(\nu)}^2+ c_{3}\|\nabla_q u\|_{L^2(\nu)}^2 \\ 
    & \ \ \ + \frac{c_{3}C_{\rm LSI}}{c} \left\|\nabla^2 u\right\|_{L^2(\nu)}^2 + \frac1c \|\nabla_q u\|_{L^2(\nu)}^2 \left[\log d + \log\left( \max_{i=1,\dots,d} \int_\cD \rme^{c|\partial_i V|} \, d\nu \right)\right].
    \end{aligned}
\end{equation}
Taking $c = \varepsilon^{-1}c_3 C_{\rm LSI}$ leads to 
\[
\left\|\nabla^2 u\right\|_{L^2(\nu)}^2 \leq \frac{1}{1-\varepsilon}\left\|\nabla_q^{*}\nabla_q u \right\|_{L^2(\nu)}^2 + L'_{d, \varepsilon}\|\nabla_q u\|_{L^2(\nu)}^2,
\]
with
\[
L'_{d, \varepsilon} := \frac{1}{1-\varepsilon}\left(c_3 + \frac{\varepsilon}{c_3 C_{\rm LSI}}\left[\log d + \log\left( \max_{i=1,\dots,d} \int_\cD \rme^{\varepsilon^{-1}c_3 C_{\rm LSI}|\partial_i V|} \, d\nu \right)\right]\right).
\]
We next set $\eps = 1/2$ and follow the same reasoning as in the end of the proof of Proposition~\ref{prop:estimate_cLa^2_invBB*} to conclude.
\end{proof}

\subsection{Linear Boltzmann equation}
\label{sec:PDMP}

We turn in this section to another paradigmatic hypocoercive dynamics, known as the linear Boltzmann equation in the community of researchers working on kinetic equations, while it corresponds to the randomized Hybrid Monte Carlo method in computational statistics and molecular simulation.

We consider the linear Boltzmann equation~\eqref{eq:Boltzmann_eq} with a more general kinetic energy~$U(v)$ than~$|v|^2/2$, for the choice
\[
k(y,v,w) = \frac{\gamma}{Z_\kappa} \rme^{-\beta U(v)}.
\]
This corresponds to the following dynamics 
\[
\frac{\partial f}{\partial t}(t,y,v) - \nabla V(y) \cdot \nabla_v f(t,y,v) + \nabla U(v) \cdot\nabla_y f(t,y,v) = \gamma \left[ \frac{\rme^{-\beta U(v)}}{Z_\kappa} \int_{\R^d} f(t,y,w) \, dw - f(t,y,v)\right].
\]
We next change the unknown function from~$f(t,y,v)$ to~$\varphi(t,y,v) = f(t,y,v) / f_\infty(y,v)$, with~$f_\infty$ the steady-state of the above equation. In fact, $f_\infty$ is the density of~\eqref{eq:canonical_U_V}, namely $f_\infty(y,v) = Z^{-1} \rme^{-\beta (V(y)+U(v))}$. We then obtain the following dynamics on~$\varphi$:
\[
\partial_t \varphi + \cLham \varphi = \gamma(\Pi_0-1)\varphi,
\]
with~$\Pi_0$ defined in~\eqref{eq:def_Pi}. This shows that the operator to consider for the longtime convergence of the linear Boltzmann equation is $\cL = \cLa + \cLs$ with $\cLa = -\cLham$ and
\begin{equation}
  \label{eq:LFD_PDMP}
  \cLs = \gamma(\Pi_0-1).
\end{equation}

\begin{remark}
  We can alternatively interpret the generator as the Fokker--Planck operator associated with a piecewise deterministic Markov process where Hamiltonian trajectories are interrupted at random exponential times by a resampling of the momenta according to~$\kappa$. This corresponds to the so-called Randomized Hybrid Monte Carlo method~\cite{BRSS17}. 
\end{remark}

The estimates on the resolvent of the generator obtained by the approach described in Section~\ref{sec:abstract} are the following.

\begin{prop}
  \label{prop:RHMC}
  Suppose that Assumption~\ref{ass:Poincare_nu} holds true and that $U$ is smooth and even, with $\rme^{-\beta U} \in L^1(\R^d)$. Then, the resolvent of the generator $\cL = \cLa + \cLs$ with $\cLa$ given by the opposite of the generator in~\eqref{eq:Lham} and $\cLs$ given by~\eqref{eq:LFD_PDMP} satisfies the following bound:
  \[
    \begin{aligned}
      \|\cL^{-1}\| & \leq \frac{2\beta\gamma}{\lambda_{\rm min}(\mathcal{M}) K_\nu^2} + \frac{2}{\gamma} \left( \frac32 + \left\| \Pi_\subplus \cLham^2 \Pi_0 \left(\cLa_{\subplus 0}^*\cLa_{\subplus 0}\right)^{-1} \right\|^2 \right).
    \end{aligned}
  \]
\end{prop}

Note that, here as well, the upper bound on the resolvent of the generator scales as $\dps \min(\gamma,\gamma^{-1})$. The scaling with respect to the dimension~$d$ can be made precise as in Sections~\ref{sec:scaling_dimension_Langevin} and~\ref{sec:scaling_dimension_Langevin_improved}.

\begin{proof}
  All the computations of Section~\ref{eq:general_estimates_Langevin} are valid upon taking $s = \gamma$. The only changes that need to be made are in the verification of Assumption~\ref{ass:technical_ass}. Note first that $\Pi_1 \cLs \Pi_1 = \gamma \Pi_1 \Pi_\subplus \Pi_1 = \gamma \Pi_1$, which immediately implies that $\|\Pi_1 \cLs \Pi_1\| = \gamma$. Next, $\cLs_{21} = 0$, so that $\cLs_{21} \cLa_{10} \left(\cLa_{\subplus 0}^*\cLa_{\subplus 0}\right)^{-1} = 0$. The final estimate then directly follows from Theorem~\ref{thm:bounds_Linv}. 
\end{proof}


\subsection{Adaptive Langevin dynamics}
\label{sec:AdL}

We show here how to apply the framework of Section~\ref{sec:abstract} to Adaptive Langevin dynamics, which is a Langevin dynamics in which the friction is a dynamical variable~$\xi \in \mathbb{R}$ following some Nos\'e--Hoover feedback dynamics. Therefore, $x = (q,p,\xi)$ and $\cX = \cD \times \mathbb{R}^d \times \mathbb{R}$. We consider for simplicity the case when
\begin{equation}
  \label{eq:quadratic_U}
  U(p) = \frac{|p|^2}{2} = \frac12 \sum_{i=1}^d p_i^2.
\end{equation}
After some suitable normalization (see~\cite{LSS19}), the invariant measure of the dynamics reads
\[
\mu(dq \, dp \, d\xi) = Z^{-1} \exp\left(-\beta \left [\frac{|p|^2}{2} + V(q) + \frac{\xi^2}{2} \right]\right) \, dq \, dp \, d\xi,
\]
and the generator is, for some $\varepsilon > 0$,
\newcommand{\cLNH}{\cL_{\rm NH}}
\[
  \cL = \cLham + \gamma \cLFD + \frac{1}{\varepsilon} \cLNH,
\]
with $\cLham$ and $\cLFD$ defined respectively in~\eqref{eq:Lham} and~\eqref{eq:LFD_Langevin}, and 
\[
  \cLNH = \left (|p|^{2} - \frac{d}{\beta} \right ) \partial_{\xi} - \xi\, p^T \nabla_{p} = \frac{1}{\beta^2} \left((\partial_\xi-\partial_\xi^*)\nabla_p^*\nabla_p + \Delta_p^*\partial_\xi - \Delta_p\partial_\xi^* \right).
\]
Therefore, $\cLs = \gamma \cLFD$ as for standard Langevin dynamics, while $\cLa = \cLham + \varepsilon^{-1} \cLNH$. The framework introduced in Section~\ref{sec:abstract} allows to retrieve the scaling on the resolvent of the generator found in~\cite{LSS19} using the techniques from~\cite{Herau06,DMS15}.

\begin{prop}
  Suppose that Assumption~\ref{ass:Poincare_nu} holds true. Then there exists $C \in \mathbb{R}_+$ such that
  \[
    \forall \gamma,\varepsilon>0, \qquad \|\cL^{-1}\| \leq C\max\left(\gamma\varepsilon^2,\gamma,\frac1\gamma,\frac{1}{\gamma\varepsilon^2}\right).
\]
\end{prop}

\begin{proof}
Assumptions~\ref{ass:structural_Pi}, \ref{ass:A_inv} and~\ref{ass:reversibility} hold true as for Langevin dynamics. A simple computation next shows that
\[
\cLa_{\subplus 0}^*\cLa_{\subplus 0} = \left (\frac{2d}{\beta^2 \varepsilon^2} \partial_{\xi}^{*}\partial_{\xi} + \frac1\beta \nabla_{q}^{*}\nabla_{q} \right ) \Pi_0,
\]
which, by tensorization of Poincar\'e inequalities, implies that Assumption~\ref{ass:macro_coercivity} holds with
\[
  a^2 = \frac1\beta \min\left(\frac{2d}{\varepsilon^2},K_\nu^2\right).
\]
Let us next turn to Assumption~\ref{ass:technical_ass}. Computations similar to the ones performed for Langevin dynamics (see the proof of Corollary~\ref{cor:Lang}) show that $\Pi_1 \cLFD \Pi_1$ is bounded and $\cLs_{21} \cLa_{12} \Pi_0  \left(B^* B\right)^{-1} = 0$ since $\cLs_{21} = \gamma \Pi_2 \cLFD \Pi_1 = 0$ for the choice~\eqref{eq:quadratic_U}. From the computations in the proof of~\cite[Lemma~2.7]{LSS19}, there exists $R \in \R_+$ such that 
\[
\left\| \Pi_\subplus\cLham^2 \Pi_0  \left(\cLa_{\subplus 0}^*\cLa_{\subplus 0}\right)^{-1} \right\| \leq R \max\left(1,\frac1\varepsilon\right).
\]
An application of Theorem~\ref{thm:bounds_Linv} then allows to conclude.
\end{proof}
  
\appendix
\section{Generalized Poincar\'e-type inequalities}
\label{sec:gen_Poincare}

We show here how to obtain Poincar\'e-type inequalities, as recently derived in~\cite{AM19} and revisited in~\cite{CLW19}, relying on the algebraic framework presented in Section~\ref{sec:abstract}. We assume in all this section that $\cLa \mathbf{1} = 0$ and that $\cLs$ is a negative operator. Let us emphasize that some arguments are formal. The main interest of our presentation in our opinion is to make explicit the algebraic framework behind the estimates obtained in~\cite{AM19,CLW19} for Langevin dynamics, in order to extend the approach to other hypocoercive dynamics (as done in~\cite{LW20} for various models of piecewise deterministic Markov processes, as the one considered in Section~\ref{sec:PDMP}). In particular, the framework used here allows to simplify some algebraic manipulations, as for instance in the proof of~\cite[Theorem~2]{CLW19}.

We first show how to obtain Poincar\'e-type inequalities using the antisymmetric part of the generator, in a static setting where only spatial degrees of freedom are considered. We then give a somewhat abstract account of Poincar\'e-type inequalities in a space-time setting and recall how they are used to prove the exponential decay of the evolution semigroup. We do not provide new results in the space-time setting, but merely present the results of~\cite{AM19,CLW19} in the language of this work.

\paragraph{Poincar\'e-type inequality using the antisymmetric part of the generator.}
We start by proving the following Poincar\'e-type estimate. Recall that $\Pi_\subplus = 1-\Pi_0$. We consider, in the result below, a slight modification of Assumption~\ref{ass:technical_ass}. 

\begin{prop}
\label{prop:Poincare_static}
  Suppose that Assumptions~\ref{ass:structural_Pi} and~\ref{ass:macro_coercivity} hold true, and that the operators~$(1-\cLs)^{1/2} \Pi_{1}$ and~$\Pi_\subplus \cLa^2 \Pi_0 \left(\cLa_{\subplus 0}^*\cLa_{\subplus 0}\right)^{-1}$ are bounded. Then, 
  \begin{equation}
    \label{eq:poincare_abstract}
    \forall f \in C_\mathrm{c}^\infty(\cX), \qquad \left\| f - \langle f,\mathbf{1}\rangle \right\| \leq C_1 \| (1-\Pi_0)f \| + C_2 \| (1-\cLs)^{-1/2} \cLa f \|,
  \end{equation}
  with 
  \begin{equation}
    \label{eq:expression_C1_C2}
    C_1 = 1+\left\|\Pi_\subplus \cLa^2 \Pi_0 \left(\cLa_{\subplus 0}^*\cLa_{\subplus 0}\right)^{-1}\right\| < +\infty,
    \qquad
    C_2 = \left\| (1-\cLs_{\subplus\subplus})^{1/2} \cLa_{\subplus 0}\left(\cLa_{\subplus 0}^*\cLa_{\subplus 0}\right)^{-1} \right\| < +\infty.
  \end{equation}
\end{prop}

Note that, for Langevin dynamics, $\left\| (1-\cLs)^{-1/2} \cdot \right\|$ is a norm equivalent to the canonical norm on $L^2(\nu,H^{-1}(\kappa))$ (since $\cLs = \cLFD$ with $\cLFD$ given by~\eqref{eq:LFD_Langevin}), so that~\eqref{eq:poincare_abstract} corresponds to the result in~\cite[Theorem~1.2]{AM19}. The assumption on~$(1-\cLs)^{1/2} \Pi_{1}$ is automatically satisfied for operators associated with linear Boltzmann dynamics, or Langevin dynamics with quadratic kinetic energies (since~$\cLs \Pi_1$ is proportional to~$\Pi_1$ in these cases and hence bounded, see the proof of Corollary~\ref{cor:Lang}).

\begin{proof}
To prove~\eqref{eq:poincare_abstract}, we first note that, since $\| f \| \leq \| (1-\Pi_0)f \| + \| \Pi_0 f \|$, it is sufficient to establish that there exists $\widetilde{C}_1 \in \mathbb{R}_+$ such that, for any $f \in C_\mathrm{c}^\infty(\cX)$,
\begin{equation}
  \label{eq:poincare_abstract_Pi}
  \left\| \Pi_0 f - \langle f,\mathbf{1}\rangle \right\| \leq \widetilde{C}_1 \| (1-\Pi_0)f \| + C_2 \| (1-\cLs)^{-1/2} \cLa f \|.
\end{equation}
The inequality~\eqref{eq:poincare_abstract} then holds with $C_1 = 1 + \widetilde{C}_1$. To obtain~\eqref{eq:poincare_abstract_Pi}, we assume without loss of generality that $\langle f,\mathbf{1}\rangle_{L^2(\mu)} = 0$ and compute
\[
  \begin{aligned}
    \| \Pi_0 f \|^2 & = \langle \cLa_{\subplus 0} f , \cLa_{\subplus 0} \left(\cLa_{\subplus 0}^* \cLa_{\subplus 0}\right)^{-1} \Pi_0 f \rangle  \\
    & = \langle \cLa \Pi_0 f , \cLa_{\subplus 0} \left(\cLa_{\subplus 0}^* \cLa_{\subplus 0}\right)^{-1} \Pi_0 f \rangle  \\
    & = \langle \cLa f , \cLa_{\subplus 0} \left(\cLa_{\subplus 0}^* \cLa_{\subplus 0}\right)^{-1} \Pi_0 f \rangle - \langle \cLa (1-\Pi_0) f , \cLa_{\subplus 0} \left(\cLa_{\subplus 0}^* \cLa_{\subplus 0}\right)^{-1} \Pi_0 f \rangle \\
    & \leq \|(1-\cLs)^{-1/2}\cLa f\| \left\|(1-\cLs)^{1/2} \cLa_{\subplus 0} \left(\cLa_{\subplus 0}^* \cLa_{\subplus 0}\right)^{-1} \Pi_0 f \right\| \\
    & \ \ \ + \|(1-\Pi_0)f\| \left\| \Pi_\subplus \cLa^2 \Pi_0 \left(\cLa_{\subplus 0}^* \cLa_{\subplus 0}\right)^{-1} \Pi_0 f \right\|. 
  \end{aligned}
\]
Note that the operator $(1-\cLs)^{1/2} \cLa_{\subplus 0} \left(\cLa_{\subplus 0}^* \cLa_{\subplus 0}\right)^{-1}$ can be written as the composition of $(1-\cLs)^{1/2} \Pi_1$ (bounded by assumption) and the bounded operator $\cLa_{\subplus 0} \left(\cLa_{\subplus 0}^* \cLa_{\subplus 0}\right)^{-1}$. This leads finally to~\eqref{eq:poincare_abstract}.
\end{proof}

\paragraph{Space-time Poincar\'e inequality (formal).}
We now turn to space-time domains. We fix a time~$T>0$ and consider the reference Hilbert space 
\[
\cH_T = \left\{ \varphi \in L^2(\widetilde{\mu}_T) \, \middle| \, \left\langle \varphi,\mathbf{1}\right\rangle_{L^2(\widetilde{\mu}_T)} = 0 \right\},
\]
where
\[
\widetilde{\mu}_T(dt \, dx) = \lambda_T(dt) \mu(dx), \qquad \lambda_T(dt) = \frac1T \mathbf{1}_{[0,T]}(t) \, dt.
\]
The space-time Poincar\'e inequality proved in~\cite{AM19,CLW19} is the following: For any~$T>0$, there exist $C_{1,T},C_{2,T}\in \mathbb{R}_+$ such that, for any $f \in C_\mathrm{c}^\infty([0,T] \times \cX)$,
\[
  \left\| f-\left\langle f,\mathbf{1}\right\rangle_{L^2(\widetilde{\mu}_T)}\right\|_{L^2(\widetilde{\mu}_T)} \leq C_{1,T} \| (1-\Pi_0)f \|_{L^2(\widetilde{\mu}_T)} + C_{2,T} \| (1-\cLFD)^{-1/2}\left(-\partial_t + \cLham\right) f \|_{L^2(\widetilde{\mu}_T)},
\]
with $\cLham$ and $\cLFD$ defined in~\eqref{eq:Lham} and~\eqref{eq:LFD_Langevin} respectively. 

This suggest considering the following space-time Poincar\'e inequality for generators $\cL = \cLa + \cLs$:
\begin{equation}
  \label{eq:poincare_abstract_time}
  \left\| f-\left\langle f,\mathbf{1}\right\rangle_{L^2(\widetilde{\mu}_T)}\right\|_{L^2(\widetilde{\mu}_T)} \leq C_{1,T} \| (1-\Pi_0)f \|_{L^2(\widetilde{\mu}_T)} + C_{2,T} \| (1-\cLs)^{-1/2}\left(-\partial_t + \cLa\right) f \|_{L^2(\widetilde{\mu}_T)}.
\end{equation}
Formally, this corresponds to replacing $\cL$ in the statement of Proposition~\ref{prop:Poincare_static} by $-\partial_t + \cL$, considered on the Hilbert space~$\cH_T$. The total antisymmetric part of the operator is then $-\partial_t + \cLa$. In order to derive an inequality such as~\eqref{eq:poincare_abstract_time}, it formally suffices to replace the operator~$\cLa_{\subplus 0}$ in the derivation of~\eqref{eq:poincare_abstract_Pi} by $-\partial_t \Pi_0 + \cLa_{\subplus 0}$. Let us emphasize that this operator does not correspond to the restriction of~$-\partial_t + \cLa + \cLs$ from~$\cH_0$ to~$\cH_+$ because of the term~$-\partial_t \Pi_0$. There are however some caveats in such a derivation, related to the boundary conditions in the time domain, which is why a careful treatment is required; see~\cite[Lemma~2.6]{CLW19}.

\paragraph{Exponential convergence.} 
Let us finally show how to formally deduce an exponential convergence from~\eqref{eq:poincare_abstract_time} (the complete rigorous argument would involve regularizations and truncations in order to give a meaning to all time derivatives) when Assumption~\ref{ass:A_inv} holds true. Recall that scalar products and associated norms are by default the ones on~$L^2(\mu)$, unless explicitly mentioned otherwise.

In order to highlight the fact that it is possible to obtain lower bounds on the convergence rate with respect to some parameter in the dynamics, we consider the case when the generator is of the form
\[
  \cL = \cLa + \gamma \cLs,
  \]
  for some parameter~$\gamma > 0$.
Consider $\varphi(t) = \rme^{t \cL} \varphi_0$ for a given function $\varphi_0 \in \cH$. Note first that $\langle \varphi(t),\mathbf{1}\rangle = 0$ for all $t \geq 0$, so that $\langle \varphi,\mathbf{1}\rangle_{L^2(\widetilde{\mu}_T)} = 0$. Moreover, $t \mapsto \|\varphi(t)\|$ is non-increasing since
\[
\frac{d}{dt}\left(\frac12 \|\varphi(t)\|^2\right) = \gamma \left\langle\varphi(t),\cLs\varphi(t)\right\rangle \leq 0.
\]
Finally, from the evolution equation satisfied by~$\varphi(t)$,
\[
-\gamma \cLs \varphi(t) = (-\partial_t + \cLa) \varphi(t) = \Pi_\subplus(-\partial_t + \cLa) \varphi(t),
\]
which implies, upon applying $(-\cLs)^{-1/2}$ to both sides of the above equality (which is indeed possible in view of Assumption~\ref{ass:A_inv}), that
\[
\begin{aligned}
  -\gamma \left\langle\varphi(t),\cLs\varphi(t)\right\rangle & = \gamma \left\| (-\cLs)^{1/2} \varphi(t) \right\|^2 = \frac1\gamma \left\| (-\cLs)^{-1/2}(-\partial_t + \cLa)\varphi(t)\right\|^2\\
  & \geq \frac1\gamma \left\| (1-\cLs)^{-1/2}(-\partial_t + \cLa)\varphi(t)\right\|^2.
\end{aligned}
\]
Therefore, for any $\eta \in (0,1)$,
\[
  \begin{aligned}
    &   \|\varphi(T)\|^2 - \|\varphi(0)\|^2  = 2\gamma T \left\langle\varphi,\cLs\varphi \right\rangle_{L^2(\widetilde{\mu}_T)} \\
  & \qquad \leq - 2\gamma T(1-\eta)s \left\| \Pi_\subplus \varphi \right\|_{L^2(\widetilde{\mu}_T)}^2 - \frac{2T\eta}{\gamma} \left\| \left(1-\cLs\right)^{-1/2}(-\partial_t + \cLa) \varphi \right\|_{L^2(\widetilde{\mu}_T)}^2.
  \end{aligned}
\]
The choice $\eta = \gamma^2 s C_{2,T}^2/(\gamma^2 sC_{2,T}^2 + C_{1,T}^2)$ leads to
\[
\begin{aligned}
& \|\varphi(T)\|_{L^2(\mu)}^2 - \|\varphi(0)\|_{L^2(\mu)}^2 \\
& \leq -\frac{2\gamma sT}{\gamma^2 sC_{2,T}^2 + C_{1,T}^2} \left( C_{1,T}^2 \left\| (1-\Pi_0) \varphi(t) \right\|_{L^2(\widetilde{\mu}_T)}^2 + C_{2,T}^2\left\| \left(1-\cLs\right)^{-1/2}(-\partial_t + \cLa) \varphi(t) \right\|_{L^2(\widetilde{\mu}_T)}^2\right) \\
& \leq -\frac{\gamma sT}{\gamma^2 sC_{2,T}^2 + C_{1,T}^2} \left( C_{1,T} \left\| (1-\Pi_0) \varphi(t) \right\|_{L^2(\widetilde{\mu}_T)} + C_{2,T}\left\| \left(1-\cLs\right)^{-1/2}(-\partial_t + \cLa) \varphi(t) \right\|_{L^2(\widetilde{\mu}_T)}\right)^2 \\
& \leq -\frac{\gamma sT}{\gamma^2 sC_{2,T}^2 + C_{1,T}^2} \|\varphi\|_{L^2(\widetilde{\mu}_T)}^2 \leq -\frac{\gamma s T}{\gamma^2 sC_{2,T}^2 + C_{1,T}^2} \|\varphi(T)\|_{L^2(\mu)}^2,
\end{aligned}
\]
where we used the Poincar\'e inequality~\eqref{eq:poincare_abstract_time}, and the fact that $t \mapsto \|\varphi(t)\|_{L^2(\mu)}$ is non-increasing in the last step. By rewriting this final inequality as
\[
\|\varphi(T)\|_{L^2(\mu)}^2 \leq \alpha_T \|\varphi(0)\|_{L^2(\mu)}^2, \qquad 0 \leq \alpha_T = \left(1+\frac{\gamma sT}{\gamma^2 sC_{2,T}^2 + C_{1,T}^2}\right)^{-1} < 1,
\]
we can then conclude to an exponential convergence as in~\cite{AM19,CLW19}. It is also possible to make explicit the scaling with respect to the dimension and the friction coefficient~$\gamma$ in this setting, as done for Langevin dynamics in Section~\ref{sec:application_Langevin}. More precisely, for $T>0$ fixed, it holds
\[
\alpha_T \mathop{\sim}_{\gamma\to 0} \rme^{-\gamma s T/C_{1,T}^2}, \qquad \alpha_T \mathop{\sim}_{\gamma\to +\infty} \rme^{-T/(\gamma C_{2,T}^2)},
\]
which shows that the convergence rate of the semigroup~$\rme^{t \cL}$ is indeed of order~$\min(\gamma,\gamma^{-1})$.

\begin{remark}
  \label{rmk:false}
  If the constants~$C_{1,T},C_{2,T}$ in~\eqref{eq:poincare_abstract_time} were uniformly bounded by values $C_1,C_2>0$ with respect to~$T > 0$, it would be possible to take the limit $T \to 0$ in the above estimates. This would lead to
\[
\frac{d}{dt} \left( \|\varphi(t)\|_{L^2(\mu)}^2 \right) \leq -\frac{\gamma s}{\gamma^2 sC_{2}^2 + C_{1}^2} \|\varphi(t)\|_{L^2(\mu)}^2,
\]
and hence to an exponential convergence without prefactor, which is clearly false since this would imply the coercivity of the generator for the canonical scalar product on~$L^2(\mu)$. In~\cite{CLW19}, the constant~$C_{1,T}$ in~\eqref{eq:poincare_abstract_time} involves a term scaling as~$1/T$, which prevents taking the limit~$T \to 0$. It can be traced back to the estimates on~$\psi_2'$ in~\cite[Lemma~2.6]{CLW19}.
\end{remark}


\paragraph{Acknowledgements.}
We thank David Herzog for fruitful discussions on~\cite{AM19,CLW19}, Francis Nier for pointing out~\cite{SZ07}, and the Heilbronn Institute for Mathematical Research in Bristol, UK, where this work was first presented in a workshop on hypocoercivity (3-4 March 2020). The work of A.L and G.S. was funded in part from the European Research Council (ERC) under the European Union's Horizon 2020 research and innovation programme (grant agreement No 810367). The work of G.S. was also funded by the Agence Nationale de la Recherche, under grant ANR-19-CE40-0010-01 (QuAMProcs). The work of M.F. was funded by the ANR, under grants ANR-17-CE40-0030 (EFI) and ANR-18-CE40-006 (MESA), as well as ANR-11-LABX-0040-CIMI within the program ANR-11-IDEX-0002-02. This work was done while M.F. was a guest of the teams Mokaplan and Matherials at the INRIA Paris, whose support is gratefully acknowledged. 


\end{document}